\documentclass[a4paper,12pt,reqno]{amsart}

\usepackage{amsmath, amsfonts, amssymb, amsthm, mathrsfs, mathtools, mathalfa, setspace}
\usepackage{caption}
\usepackage{comment}
\usepackage{yfonts}
\usepackage{enumerate}
\usepackage{enumitem}
\usepackage{stmaryrd}
\usepackage{fancyhdr}
\usepackage{bm}
\usepackage[utf8]{inputenc}
\usepackage[pdfencoding=auto,hyperfootnotes=false,backref=page]{hyperref}
\usepackage{tikz,tikz-cd}
\usepackage[hang,flushmargin]{footmisc}
\usepackage{xr}
\usepackage{extarrows}
\usepackage{todonotes}
\usepackage{faktor}
\usepackage{xfrac}
\usepackage{times}
\usepackage{color}
\usepackage{graphicx}
\usetikzlibrary{babel}
\usepackage{xurl}

\usetikzlibrary{arrows}
\usetikzlibrary{positioning, arrows.meta}

\usepackage{tabularx}
\usepackage{multirow}

\makeatletter
\def\@tocline#1#2#3#4#5#6#7{\relax
  \ifnum #1>\c@tocdepth 
  \else
    \par \addpenalty\@secpenalty\addvspace{#2}%
    \begingroup \hyphenpenalty\@M
    \@ifempty{#4}{%
      \@tempdima\csname r@tocindent\number#1\endcsname\relax
    }{%
      \@tempdima#4\relax
    }%
    \parindent\z@ \leftskip#3\relax \advance\leftskip\@tempdima\relax
    \rightskip\@pnumwidth plus4em \parfillskip-\@pnumwidth
    #5\leavevmode\hskip-\@tempdima
      \ifcase #1
       \or\or \hskip 1em \or \hskip 2em \else \hskip 3em \fi%
      #6\nobreak\relax
    \dotfill\hbox to\@pnumwidth{\@tocpagenum{#7}}\par
    \nobreak
    \endgroup
  \fi}
\makeatother

\newtheorem{theorem}{Theorem}[section]
\newtheorem{lemma}[theorem]{Lemma}

\newtheorem{corollary}[theorem]{Corollary}
\newtheorem{proposition}[theorem]{Proposition}

\theoremstyle{definition}
\newtheorem{defn}[theorem]{Definition}
\newtheorem{remark}[theorem]{Remark}
\newtheorem{example}[theorem]{Example}

\newcommand*{\sbr}[1]{\scalebox{0.8}{$(#1)$}}
\newcommand*{\db}[1]{\llbracket #1\rrbracket}

\newcommand{\mc}{\mathcal}

\newcommand{\mf}{\mathbf}
\newcommand{\mb}{\mathbb}

\newcommand{\wh}{\widehat}
\newcommand{\wt}{\widetilde}

\newcommand{\id}{\mathrm{id}}

\newcommand{\stab}{\mathrm{Stab}}
\newcommand{\aut}{\mathrm{Aut}}

\newcommand{\ab}{\mathrm{Z}}

\newcommand{\tran}{\mathrm{\Theta}}

\newcommand{\poly}{\mathrm{poly}}

\newcommand{\q}{\mathrm{c}}
\newcommand{\ns}{\mathrm{X}}
\newcommand{\nss}{\mathrm{Y}}
\newcommand{\co}{\circ\hspace{-0.02cm}}
\newcommand{\cu}{\mathrm{C}}

\begin{document}

\vspace*{-0.8cm}

\title[The Jamneshan--Tao conjecture for finite abelian groups of bounded rank]{The Jamneshan--Tao conjecture\\ for finite abelian groups of bounded rank}

\author{Pablo Candela}
\address{Instituto de Ciencias Matem\'aticas, Calle Nicol\'as Cabrera 13-15, Madrid 28049, Spain}
\email{pablo.candela@icmat.es}

\author{Diego Gonz\'alez-S\'anchez}
\address{Universit\'e Paris Cit\'e, Sorbonne Universit\'e, CNRS, IMJ-PRG, F-75013 Paris, France}
\email{gonzalezsanchez@imj-prg.fr}

\author{Bal\'azs Szegedy}
\address{HUN-REN Alfr\'ed R\'enyi Institute of Mathematics\\ 
Re\'altanoda utca 13-15\\
Budapest, Hungary, H-1053}
\email{szegedyb@gmail.com}

\maketitle

\begin{abstract} 
We confirm the Jamneshan--Tao conjecture for finite abelian groups of rank at most a fixed integer $R$ (i.e.\ finite abelian groups generated by at most $R$ elements), by proving an inverse theorem for 1-bounded functions of non-trivial Gowers norm on such groups, concluding that such a function must correlate non-trivially with a nilsequence of bounded complexity.
\end{abstract}

\section{Introduction}

\subsection{Background}\hfill\\
\noindent The field of higher-order Fourier analysis arose in arithmetic combinatorics,  from efforts to extend  the Fourier-analytic methods used in the proof of Roth’s theorem to the full generality of Szemerédi’s theorem on arithmetic progressions. This field was initiated by Gowers in \cite{GSz}, and it developed principally around the  \emph{uniformity norms} (or \emph{Gowers norms}), which involve averaging a function over combinatorial cubes, as we now recall.

\begin{defn}\label{def:Uk}
Let $\ab$ be a finite abelian group and let $k\ge 2$ be an integer. The \emph{Gowers uniformity norm of order $k$ on $\ab$}, or \emph{$U^k$-norm}, is defined on the vector space of functions $f:\ab\to \mb{C}$ by the  formula\footnote{Where $\mc{C}^n$ stands for complex conjugation applied $n$ times, and $|v|:=v_1+\cdots+v_k$ for $v\in \{0,1\}^k$.}
$\|f\|_{U^k}^{2^k}:=\mb{E}_{x,h_1,\ldots,h_k\in\ab}\prod_{v\in \{0,1\}^k} \mc{C}^{|v|} f(x+v_1h_1+\cdots +v_kh_k)$.
\end{defn}

\noindent In the same decade, Gowers developed an approach to hypergraph regularity based on  norms which are closely related to the uniformity norms (known as \emph{octahedral norms} or \emph{box norms}) \cite{Gow-3hyps,Gow-khyps}.
These interconnected developments placed higher-order Fourier analysis at the center of a deep and far-reaching new paradigm, in which useful and versatile notions of quasirandomness and regularity for  structures of various types (including subsets of abelian groups, but also graphs and hypergraphs) can be analyzed and quantified using norms involving specific key configurations (e.g.\ Gowers norms involving combinatorial cubes in abelian groups, or box norms involving octahedra in hypergraphs).\footnote{This paradigm extends also into ergodic theory, via the Host--Kra seminorms analogous to Gowers norms \cite{HKbook}.} Central examples of this paradigm are the \emph{regularity lemmas} in which the quasirandomness is expressed as smallness in such a norm (see \cite[Theorem 8.10]{Gow-3hyps} and \cite[Theorem 1.2]{GTarith} for examples in hypergraph theory and arithmetic combinatorics respectively). A related family of examples is that of \emph{inverse theorems} for such norms. 

Inverse theorems emerged in arithmetic combinatorics, focusing on the Gowers norms\footnote{There are also results in graph and hypergraph theory that can be viewed as inverse theorems, involving the \emph{cut norm} and higher-order generalizations; see for instance \cite[Theorems 3.1 and 4.1]{Gow-3hyps}.} \cite{GT08}. Generally speaking, the idea of an inverse theorem for the $U^k$-norm on some finite abelian group $\ab$ is to deduce information of harmonic-analytic type about the structure of a function $f:\ab\to\mb{C}$, from the more combinatorial assumption that the norm $\|f\|_{U^k}$ is non-trivially large. The Gowers norms form an increasing sequence (that is,  we always have $\|f\|_{U^k}\leq \|f\|_{U^{k+1}}$), so the simplest case of such an inverse theorem concerns the smallest of these norms, namely the $U^2$-norm (for which the initial assumption of largeness is the strongest). In this case, an inverse theorem is easily deduced from the following simple formula expressing the $U^2$-norm of a function $f$ in terms of its Fourier transform\footnote{Recall that the Fourier transform $\wh{f}$ is the function $\ab\to\mb{C}$ defined by $\wh{f}(\xi)=\mb{E}_{x\in \ab} f(x)\overline{e(\xi\cdot x)}$, where $e(\theta)=\exp(2\pi i\theta)$ for any $\theta\in \mb{R}/\mb{Z}$, and $(\xi,x)\mapsto \xi\cdot x$ is a symmetric non-degenerate bilinear form $\ab\times \ab\to \mb{R}/\mb{Z}$; we refer to \cite[Section 4]{T-V} for basic background concerning Fourier analysis on finite abelian groups.}: $\|f\|_{U^2}=(\sum_{\xi\in \ab} |\wh{f}(\xi)|^4)^{1/4}$.

From this, one easily deduces the following statement, which is the aforementioned simplest example of an inverse theorem for Gowers norms: if $f:\ab\to \mb{C}$ is \emph{1-bounded}\footnote{We say that a function $f:\ab\to\mb{C}$ is 1-bounded if $|f(x)|\leq 1$ for every $x\in\ab$.} and $\|f\|_{U^2}\geq \delta>0$, then there is a Fourier character $x\mapsto e(\xi\cdot x)$ on $\ab$ such that $|\mb{E}_{x\in \ab} f(x)\overline{e(\xi\cdot x)}|\geq \delta^2$.

When we try to obtain an analogous statement for the $U^k$-norm for $k>2$ (thus seeking an inverse theorem for this norm), we soon find that Fourier characters do not suffice to obtain the conclusion,\footnote{See for instance \cite[\S 4]{GSz}, with the example of the function $f(x)=e(x^2/N)$ on $\mb{Z}/N\mb{Z}$, which has maximally large $U^3$-norm but no large Fourier coefficient; see also \cite[Exercise 11.1.12]{T-V}.} and we come to a central problem in this topic: to identify generalizations of Fourier characters yielding useful inverse theorems.

The initial progress on this problem concerned specific families of finite abelian groups. 

The first version of a useful inverse theorem, valid for the $U^k$-norm for every $k\geq 2$, was proved by Green, Tao, and Ziegler \cite{GTZ}, essentially for the family of finite cyclic groups $\mb{Z}/N\mb{Z}$, where $N$ is usually assumed to be prime (this is often called the \emph{integer setting}). This result had powerful applications in number theory (in particular for counting various kinds of linear configurations in the prime numbers \cite{GT-lin}). Adequate generalizations of Fourier characters in this setting turned out to be the functions known as \emph{nilsequences}. These generalizations consist mainly in replacing the circle group $\mb{R}/\mb{Z}$ (underpinning classical Fourier characters) with more general \emph{filtered nilmanifolds}, and replacing  homomorphisms $\ab\to\mb{R}/\mb{Z}$ with \emph{polynomial maps} from $\ab$ to such  nilmanifolds. Let us recall the formal definition. 
\begin{defn}\label{def:nilseqs}
Let $\ab$ be a finite abelian group. A (polynomial) \emph{nilsequence of degree $k$} on $\ab$ is a function $\ab\to\mb{C}$ of the form $x\mapsto F(g(x))$ constructed as follows. There is a filtered nilmanifold $(G/\Gamma,G_\bullet)$ of degree $k$ (see Definition \ref{def:nilmanifold}) and a polynomial map $\varphi:\ab\to G$ relative to the filtration $G_\bullet$ (see\footnote{Equivalently, a map $\varphi:\ab\to G$ is polynomial relative to the given filtration $G_\bullet=(G_i)_{i\geq 0}$ if for any $t_1,\ldots,t_i,x\in \ab$ we have  $\partial_{t_1}\cdots\partial_{t_i}f(x)\in G_i$, where for any $t\in \ab$ we define the discrete derivative $\partial_t h:\ab\to G$, $x\mapsto h(x)^{-1}h(x+t)$.} Definition \ref{def:nsmorphisms}) such that $g:\ab\to G/\Gamma$ is the map $g(x)=\varphi(x)\Gamma$, and $F$ is a continuous map $G/\Gamma\to \mb{C}$. For this notion to be non-trivial, we usually require the nilsequence to have \emph{bounded complexity}. This requirement consists in first fixing a \emph{complexity notion} for filtered nilmanifolds, that is, an arbitrary ordering of the countable set of (isomorphism classes of) nilmanifolds, and then declaring that a nilsequence $F(g(x))$ has \emph{complexity at most $M$} if $G/\Gamma$ has position at most $M$ in the ordering and the Lipschitz norm of $F$ is at most $M$.
\end{defn}

\noindent Another family of abelian groups that was central to the early developments of inverse theorems is that of vector spaces $\mb{F}_p^n$ over a finite field of fixed prime characteristic $p$ (this is often called the \emph{finite-field setting}) \cite{BTZ,T&Z-Low}. Here, adequate  generalizations of Fourier characters for the $U^k$-norm turned out to be the \emph{polynomial phase functions of degree $k-1$}, i.e.\ functions of the form $x\in \mb{F}_p^n$ $\mapsto$ $e(P(x))$ where $P:\mb{F}_p^n\to \mb{R}/\mb{Z}$ is a polynomial\footnote{This means that any $k$-fold discrete derivative $\partial_{t_1}\cdots \partial_{t_k}P$ vanishes everywhere.} of degree at most $k-1$. Note that polynomial phase functions of degree $k-1$ are polynomial nilsequences\footnote{Indeed the underlying filtered nilmanifold here is $\mb{R}/\mb{Z}$ equipped with the degree-$(k-1)$ filtration $G_\bullet$ on $\mb{R}$ where $G_i=\mb{R}$ for all $i\in [0,k-1]$ (this coincides with the structure  denoted $\mc{D}_{k-1}(\mb{R}/\mb{Z})$ recalled in Definition \ref{def:Dkab}). Note that $P:\mb{F}_p^n\to\mb{R}/\mb{Z}$ is polynomial of degree $k-1$ in the sense of the previous footnote if and only if $P$ is polynomial in the sense of Definition \ref{def:nsmorphisms}, for  $\mb{R}/\mb{Z}$ equipped with the mentioned filtration $G_\bullet$ and $\mb{F}_p^n$ equipped with the lower-central series filtration.} of degree $k-1$.

In the two initial settings mentioned above, the proofs of the corresponding inverse theorems were different conceptually. A deeper understanding of this topic was then sought, especially in order to obtain a more general inverse theorem, valid for all finite abelian groups, which would unify these two initial settings. An important part of this effort centered on identifying the fundamental properties that must be satisfied by any set equipped with  combinatorial cubic structures, abstracting in particular from the cubes involved in Gowers norms. Such a program had already been initiated by Host and Kra in \cite{HKparas}, and this led to the theory of \emph{nilspaces}  (see Definition \ref{def:nilspace} below). This theory, pioneered in \cite{CamSzeg} (see also  \cite{Cand:Notes1,Cand:Notes2} for a more detailed account), provided a framework in which a general inverse theorem could indeed be obtained. For instance, an inverse theorem for all finite abelian groups was established in \cite[Theorem 5.2]{CSinverse}, from which the inverse theorems from the integer setting \cite{GTZ} and the finite-field setting \cite{T&Z-Low} could both be deduced (see \cite{CSinverse} and \cite{CGSS}). In this theorem, the generalizations of Fourier characters, known as \emph{nilspace polynomials}, are functions of the form $F(g(x))$ similar to nilsequences, except that instead of $g$ being a polynomial map into a nilmanifold, it is a \emph{nilspace morphism} into a \emph{compact and finite-rank} (\textsc{cfr}) \emph{nilspace}  (see Definitions \ref{def:CFR} and \ref{def:nsmorphisms}). 

While they are similar to nilsequences, nilspace polynomials are strictly more general objects. Indeed, while \textsc{cfr} nilspaces can always be described in terms of nilmanifolds (for a recent example see \cite[Theorem 1.4]{CGSS-projnil}), it is also known since the beginning of this theory that some \textsc{cfr} nilspaces are \emph{not} nilmanifolds (see \cite[Example 6]{HKparas}). It is then natural to wonder whether polynomial nilsequences suffice as generalizations of Fourier characters (among the more general nilspace polynomials) to obtain inverse theorems for Gowers norms on all finite abelian groups. An interesting conjecture of Jamneshan and Tao posits that this is indeed the case \cite[Conjecture 1.11]{J&T}. This conjecture is the main focus of this paper.

Recently, several works have made progress towards a proof of the Jamneshan--Tao conjecture, including an inverse theorem for finite abelian groups in terms of projected nilsequences \cite{CGSS-projnil}, and a proof of the conjecture in the case of groups of bounded exponent \cite{JST-bndexp}. 

We refer to all the aforementioned works for further background on inverse theorems for the Gowers norms and on higher-order Fourier analysis more generally. We also refer to works in the quantitative direction of improving the bounds for the inverse theorem, for specific classes of finite abelian groups \cite{LSS, Manners, Mili2}, or for specific Gowers norms on all such groups \cite{Mili-Gen}.

\subsection{Main results}\hfill\\
\noindent Recall that the rank of a finite abelian group is the minimum cardinality of a generating subset of the group. The main result of this paper is the following inverse theorem for Gowers norms, proving the Jamneshan--Tao conjecture in the case of abelian groups of bounded rank. 

\begin{theorem}\label{thm:bndrankinv-intro}
For any $k,R\in \mb{N}$ and $\delta>0$, there is $\varepsilon>0$ and a finite collection $\mc{N}_{k,R,\delta}$ of degree-$k$ filtered nilmanifolds $G/\Gamma$, each equipped with a smooth Riemannian metric and with connected and simply-connected ambient group $G$, such that the following holds. For any  finite abelian group $\ab$ of rank at most $R$, and any 1-bounded function $f:\ab\to\mb{C}$ with $\|f\|_{U^{k+1}}\geq \delta$, there exists $G/\Gamma\in \mc{N}_{k,R,\delta}$, a Lipschitz 1-bounded function $F:G/\Gamma\to\mb{C}$ of Lipschitz norm $O_{k,R,\delta}(1)$, and a polynomial map $g:\ab\to G/\Gamma$, such that $|\mb{E}_{x\in \ab} f(x)\overline{F(g(x))}| \geq \varepsilon$.
\end{theorem}
\noindent Our proof of Theorem \ref{thm:bndrankinv-intro} can be divided into two main steps. To summarize these steps in what follows, we will use basic terminology from nilspace theory, providing references to more detailed definitions and discussion in Section \ref{sec:prelims} below.

The first step in the proof consists in applying the general inverse theorem mentioned in the previous subsection, namely \cite[Theorem 5.2]{CSinverse}, and showing that, in the bounded-rank setting, the resulting nilspace is what we call a \emph{quasitoral} nilspace. To define these nilspaces, let us first recall that every $k$-step compact nilspace $\ns$ determines corresponding \emph{structure groups} $\ab_i(\ns)$, $i\in [k]$, which are compact abelian groups (see Definition \ref{def:structgps}). We say that $\ns$ is \emph{compact finite-rank} ( \textsc{cfr}) if the dual groups $\wh{\ab_i(\ns)}$ are all finitely generated (see Definition \ref{def:CFR}). In particular, a \textsc{cfr} nilspace $\ns$ is \emph{toral} if its structure groups are all tori, i.e.\ groups of the form $\mb{R}^n/\mb{Z}^n$ for some $n\in \mb{Z}_{\geq 0}$ (see \cite[\S 2.9.1]{Cand:Notes2}).

\begin{defn}[Quasitoral nilspaces]\label{def:quasitoral}
We say that a $k$-step \textsc{cfr} nilspace $\ns$ is \emph{quasitoral} if for every $i\in [2,k]$ the structure group $\ab_i(\ns)$ is a torus.
\end{defn}
\noindent We show that quasitoral nilspaces are disjoint unions of toral nilspaces (which are connected  nilmanifolds). In the case of the quasitoral nilspace $\ns$ obtained by applying \cite[Theorem 5.2]{CSinverse}, the  number of its toral-nilspace components depends only on the complexity of $\ns$, and is therefore adequately bounded. As a consequence, we deduce that on some bounded-index subgroup $\ab'$ of the original abelian group $\ab$, some shift of the initial function $f$ correlates non-trivially with a bounded-complexity nilsequence defined on $\ab'$. This completes this first step.

The second step consists in proving that the nilsequence on $\ab'$ can always be extended to a nilsequence on the full group $\ab$, while ensuring that the latter nilsequence still has bounded complexity. To solve this extension problem in general, it is necessary to modify the nilmanifold underlying the nilsequence; indeed we give an example of a degree-2 polynomial that cannot be extended to a larger group while preserving the polynomial's target nilspace; see Example \ref{ex:nonext}.

The paper has the following outline. In Section \ref{sec:prelims} we gather some basic tools and concepts from nilspace theory and higher-order Fourier analysis. In Section \ref{sec:quasitoral} we carry out the first main step described above, and in Section \ref{sec:extend} we solve the nilsequence-extension problem. The proof of Theorem \ref{thm:bndrankinv-intro} (including the control on the final correlation bound), is given in Section \ref{sec:mainproof}.

\smallskip
 
\noindent \textbf{Acknowledgments.} This work was supported by project \begin{small}PID2024-156180NB-I00\end{small} funded by Spain's \begin{small}MICIU/AEI\end{small}. The second-named author was supported by \begin{small}HORIZON-MSCA-2024-PF-01\end{small}, \begin{small}AlgHOF 101202161\end{small}, funded by the European Union.\footnote{Views and opinions expressed are those of the author(s) only and do not reflect those of the EU or the European Commission. Neither the EU nor the European Commission can be held responsible for them.} The third-named author was supported by the Hungarian Ministry of Innovation and Technology NRDI Office within the framework of the Artificial Intelligence National Laboratory Program (\begin{small}MILAB, RRF-2.3.1-21-2022-00004\end{small}).

\section{Background on nilspaces}\label{sec:prelims}

\noindent In this section we gather and recall the most important notions from nilspace theory that we will need. Let us first define nilspaces and some related concepts, following \cite[\S 1]{Cand:Notes1}. 

\begin{defn} For each positive integer $n$, the \emph{discrete $n$-cube} is the set $\db{n}:=\{0,1\}^n$, and $\db{0}:=\{0\}$. For $m,n\in\mb{Z}_{\ge 0}$, a map $\phi:\db{m}\to \db{n}$ is a \emph{discrete-cube morphism} if it extends to an affine homomorphism\footnote{An affine homomorphism $f:Z_1\to Z_2$ between abelian groups $Z_1$ and $Z_2$ is a function $f(z)=g(z)+t$, where $g:Z_1\to Z_2$ is a homomorphism and $t\in Z_2$.} $f:\mb{Z}^m\to \mb{Z}^n$.
\end{defn}

\begin{defn}
For $m,n\in \mb{Z}_{\ge 0}$, a \emph{face} of dimension $m$ in $\db{n}$ is a subset $F\subset \db{n}$ defined by fixing $n-m$ coordinates, that is, a set of the form $F=\{\underline{v}\in\db{n}: v(i)=t(i), \ i\in I\}$ for some $I\subset \{1,\ldots,n\}$ with $|I|=n-m$ and $t(i)\in \{0,1\}$ for all $i\in I$. A \emph{face map} $\phi:\db{k}\to \db{n}$ is an injective morphism of discrete cubes such that $\phi(\db{k})$ is a face.
\end{defn}

\begin{defn}\label{def:nilspace} A \emph{nilspace} is a set $\ns$ equipped with a set $\cu^n(\ns)\subset \ns^{\db{n}}$ for every $n\in \mb{Z}_{\geq 0}$, satisfying the following axioms:
\setlength{\leftmargini}{0.8cm}
\begin{enumerate}
    \item (Composition) For every discrete-cube morphism $\phi:\db{m}\to\db{n}$ and every $\q\in\cu^n(\ns)$, we have $\q\co\phi \in\cu^m(\ns)$.
    \item (Ergodicity) $\cu^1(\ns)=\ns^{\db{1}}$.
    \item (Corner completion) Let $\q':\db{n}\setminus \{1^n\}$ be a function such that for every face map $\phi:\db{n-1}\to\db{n}$  with $\phi(\db{n-1})\subset \db{n}\setminus\{1^n\}$ we have $\q'\co\phi\in\cu^{n-1}(\ns)$. Then there exists $\q\in\cu^n(\ns)$ such that $\q(v)=\q'(v)$ for all $v\in \db{n}\setminus\{1^n\}$.
\end{enumerate}
\noindent We refer to the elements of $\cu^n(\ns)$ as the \emph{$n$-cubes} on $\ns$, and to maps $\q'$ as in axiom $(iii)$ as \emph{$n$-corners} on $\ns$. A cube $\q$ satisfying the conclusion of axiom $(iii)$ is called a \emph{completion} of the $n$-corner $\q'$. If for $n=k+1$ every $n$-corner on $\ns$ has a unique completion, then we say that $\ns$ is a \emph{$k$-step} nilspace. 

Finally, we say that $\ns$ is a \emph{compact nilspace} if it has a compact second-countable Hausdorff topology and for every $n\in\mb{N}$ the set $\cu^{n}(\ns)$ is closed in the product topology on $\ns^{\db{n}}$.
\end{defn}
The following notion is crucial in nilspace theory, as it defines a maximal equivalence relation such that the corresponding quotient space is a $k$-step nilspace.
\begin{defn}
Let $k\in\mb{N}$ and let $\ns$ be a nilspace. We define the relation $\sim_k$ on $\ns$ by declaring that $x\sim_k y$ if and only if $\exists\, \q_0,\q_1\in \cu^{k+1}(\ns)$ such that $\q_0(0^{k+1})=x,\ \q_1(0^{k+1})=y$ and $\q_0(v)=\q_1(v)\ \forall\, v\neq 0^{k+1}$. Let $\pi_k:\ns\to\ns/\sim_k$ be the quotient map. Then $\ns/\sim_k$ together with the cubes $\cu^n(\ns/\sim_k):=\pi_k^{\db{n}}(\cu^n(\ns))$ is a $k$-step nilspace, called the (canonical) \emph{$k$-step factor} of $\ns$.
\end{defn}
\noindent Using these canonical factors, a general $k$-step nilspace can be expressed as an iterated \emph{abelian bundle} (see \cite[Definition 3.2.17 and Theorem 3.2.19]{Cand:Notes1}), where the $i$-th factor is such a bundle over the $(i-1)$-th factor, with fibers being principal homogeneous spaces of an abelian group denoted $\ab_i(\ns)$. These abelian groups play a crucial role, so we recall their definition more formally.

\begin{defn}\label{def:structgps}
Let $\ns$ be a $k$-step nilspace $\ns$. For $i\in [k]$ the \emph{$i$-th structure group} $\ab_i(\ns)$ of $\ns$ is an abelian group such that the $i$-th nilspace factor $\ns_i$ is an \emph{abelian $\ab_i$-bundle} over $\ns_{i-1}$ with projection map the nilspace factor map $\ns_i\to\ns_{i-1}$ (see \cite[\S 3.2.3]{Cand:Notes1}). If we further assume that $\ns$ is a compact nilspace, then $\ab_i$ becomes a \emph{compact} abelian group \cite[\S 2.1.1]{Cand:Notes2} (with the relative topology when $\ab_i$ is identified with any fiber of the bundle $\ns_i$).    
\end{defn}

To study inverse theorems for the Gowers norms, we can usually focus on a class of nilspaces called \emph{compact finite-rank nilspaces} (or \textsc{cfr} nilspaces for short) thanks to \cite[Theorem 1.5]{CSinverse}. Let us recall their definition.

\begin{defn}\label{def:CFR}
A compact $k$-step nilspace $\ns$ is of \emph{finite rank} if for every $i\in[k]$ the structure group $\ab_i(\ns)$ is a compact abelian \emph{Lie} group (equivalently, the dual groups $\wh{\ab_i(\ns)}$, $i\in [k]$ are all finitely generated).
\end{defn}

In connection with the Jamneshan--Tao conjecture \cite[Conjecture 1.11]{J&T}, we are interested in studying \emph{nilmanifolds} and how they are related with \textsc{cfr} nilspaces that appear in the inverse theorem for abelian groups of bounded rank. First we need to recall the concept of a \emph{filtration}.

\begin{defn}
A \emph{filtration} on a group $G$ is a sequence $G_\bullet = (G_i)_{i=0}^\infty$ of subgroups of $G$ with $G = G_0 = G_1 \geq G_2 \geq \cdots $ and such that the commutator subgroup $[G_i, G_j]$ is included in $G_{i+j}$ for all $i,j \ge 0$.  We refer to $(G, G_\bullet)$ as a \emph{filtered group}. If $G_{k+1} = \{\mathrm{id}_G\}$ we say that the filtered group $(G, G_\bullet)$ is \emph{of degree $k$}. If $G_\bullet$ satisfies the above assumptions with $G_0=G_1$ replaced by $G_0\supseteq G_1$, then we say that $G_\bullet$ is a \emph{prefiltration}.
\end{defn}

\begin{defn}\label{def:nilmanifold}
A \emph{nilmanifold} is a quotient space $G/\Gamma$ where $G$ is a nilpotent Lie group\footnote{Note that we do \emph{not} require $G$ to be connected.} and $\Gamma$ is a discrete, cocompact subgroup of $G$. If $G_\bullet$ is a filtration on $G$ of degree at most $k$, with each $G_i$ a closed subgroup of $G$ and with each subgroup $\Gamma \cap G_i$ cocompact in $G_i$, then we call $(G/\Gamma, G_\bullet)$ a \emph{filtered nilmanifold of degree $k$}. We also assume that any nilmanifold is equipped with an arbitrary smooth Riemannian metric. Note that any two smooth Riemannian metrics on a compact nilmanifold are equivalent.
\end{defn}
\noindent Any degree-$k$ nilmanifold $G/\Gamma$ is a compact $k$-step nilspace  by \cite[Proposition 1.1.2]{Cand:Notes2}.\footnote{The proof of this result assumes that $G$ is connected, but the same proof works without this assumption.} To make this relation more explicit, let us recall the notions of \emph{Host-Kra cubes} on a filtered group, and the related concepts of group nilspace and coset nilspace.
\begin{defn}\label{def:hk-cubes}
Let $(G, G_\bullet)$ be a filtered group. Given a face $F$ of $\db{n}$ and $g\in G$, let $g^{F}\in G^{\{db{n}}$ be defined by $g^{F}(v)=g^{1_F(v)}$. The set $G$ together with the cube sets $\cu^n(G):=\{\,g^{F}:F\text{ face in }\db{n},g\in G_{\mathrm{codim}(F)}\,\}\le G^{\db{n}}$ for $n\ge 0$ is a nilspace, called the \emph{group nilspace} associated with $(G,G_\bullet)$. This nilspace is of step $k$ if and only if $G_\bullet$ has degree $k$. Finally, for any subgroup $\Gamma\leq G$ the quotient set $G/\Gamma$ equipped with the cubes $\cu^n(G/\Gamma):=\cu^n(G_\bullet)\Gamma^{\db{n}}$ is a $k$-step nilspace called a \emph{coset} nilspace.\footnote{This is proved in \cite[\S 2.3]{Cand:Notes1}.}
\end{defn}

\begin{remark}\label{rem:toral&metric}
As explained in \cite[Remark 2.9.19]{Cand:Notes2}, any toral nilspace $\ns$ is isomorphic (as a compact nilspace) to a filtered nilmanifold $(G/\Gamma,G_\bullet)$ (equipped with the associated Host--Kra cubes) where we may assume that each Lie group $G_i$ in $G_\bullet$ is connected and simply-connected. 
\end{remark}
\begin{defn}\label{def:Dkab}
Given an abelian group $G$, we denote by $\mc{D}_k(G)$ the $k$-step group nilspace associated with the filtration $G_\bullet$ where $G_i=G$ for $i\le k$ and $G_{i}=\{\id\}$ for $i>k$. We call this nilspace a \emph{degree-$k$ abelian group}.
\end{defn}

Let us recall what are the morphisms in the category of nilspaces.

\begin{defn}\label{def:nsmorphisms} Let $\ns,\nss$ be nilspaces. A map $\varphi:\ns\to\nss$ is a \emph{morphism} if for every $n\in\mb{N}$ and every $\q\in \cu^n(\ns)$ we have that $\varphi\co\q\in\cu^n( \nss)$. If both $\ns$ and $\nss$ are compact nilspaces, then we require $\varphi$ to be continuous. The set of morphisms from $\ns$ to $\nss$ is denoted by $\hom(\ns,\nss)$. In particular, polynomial maps from a filtered group $(H,H_\bullet)$ to another filter group $(G,G_\bullet)$ are precisely the morphisms between the associated group nilspaces (see  \cite[Theorem 2.2.14]{Cand:Notes1}).
\end{defn}
\noindent By a ``polynomial map" from an abelian group $\ab$ into a filtered nilmanifold $G/\Gamma$ (e.g.\ the polynomial map $g$ in Theorem \ref{thm:bndrankinv-intro}), we always mean a morphism from $\mc{D}_1(\ab)$ to the coset nilspace $G/\Gamma$.

For compact nilspaces, we can define a family of probability measures $(\mu_{\cu^n(\ns)})_{n\in \mb{N}}$ on the cube sets $\cu^n(\ns)$. The measure $\mu_{\cu^n(\ns)}$ is called the \emph{Haar measure} on $\cu^n(\ns)$, see \cite[\S 2.2.2]{Cand:Notes2}.  This notion is important because the regularity lemma for the Gowers norms \cite[Theorem 1.5]{CSinverse} gives us a morphism which \emph{approximately preserves these Haar measures}. To make this more precise, we need to introduce the concept of \emph{balance}, and this requires a couple of additional definitions. First, for a $k$-step compact nilspace $\ns$, and $n\in \mb{N}$, we equip the space of probablity measures on $\cu^n(\ns)$, which we denote by $\mc{P}(\cu^n(\ns))$, with the weak topology. It follows from \cite[\S 5]{CSinverse} and the Banach-Alaoglu theorem that this is a compact metric space. Now we can introduce the concept of a \emph{balanced} morphism (simplified for our purposes).

\begin{defn}[Balanced morphism; see Definition 5.1 in \cite{CSinverse}]
Let $\ns$ be a $k$-step compact nilspace. For each integer $n\geq 0$, fix a metric $d_n$ on the space $\mc{P}(\cu^n(\ns))$. Let $\ab$ be a finite abelian group and $\varphi:\mc{D}_1(\ab)\to \ns$ be a morphism. Then for $b>0$ we say that $\varphi$ is $b$-\emph{balanced} if for every $n\le 1/b$ we have that $d_n(\mu_{\cu^n(\mc{D}_1(\ab))}\co (\varphi^{\db{n}})^{-1},\mu_{\cu^n(\ns)})<b$.
\end{defn}

\section{The role of quasitoral nilspaces in the bounded-rank setting}\label{sec:quasitoral}
\noindent In this section we prove that, when we apply the general inverse theorem \cite[Theorem 5.2]{CSinverse} to abelian groups with a fix bound $R$ on their rank, the obtained \textsc{cfr} nilspaces (underlying the nilspace polynomial)  can be taken to be quasitoral.

To begin with, let us give the following characterization of quasitoral nilspaces.

\begin{lemma}\label{lem:spli-dep-on-1-factor}
Let $\ns$ be a $k$-step \textsc{cfr} nilspace for some $k\geq 1$. Let $\pi_1:\ns\to\ns_1\cong \mb{T}^n\times H$ be the projection to the canonical 1-step factor $\ns_1$, where $H$ is a finite abelian group and $n\in \mb{Z}_{\geq 0}$, and let $q:\ns_1\to H$ be the projection homomorphism. Then, for every $h\in H$, the preimage $(q\co \pi_1)^{-1}(h)$ is a subnilspace\footnote{By definition, the cubes on this subnilspace are the cubes on $\ns$ that take values in the subnilspace.} of $\ns$ and its structure groups are $\ab_i((q\co\pi_1)^{-1}(h))=\ab_i(\ns)$ for $i\ge 2$ and $\ab_1((q\co\pi_1)^{-1}(h))=\mb{T}^n$. In particular,  if for \emph{any} $h\in H$ the nilspace $(q\co \pi_1)^{-1}(h)$ is toral, then $\ns$ is quasitoral, and conversely, if $\ns$ is quasitoral then for \emph{every} $h\in H$ the nilspace $(q\co \pi_1)^{-1}(h)$ is toral.
\end{lemma}

\begin{proof}
For any $h\in H$, by \cite[Lemma 3.2]{CGSS} we have that $(q\co\pi_1)^{-1}(h)$ is a subnilspace of $\ns$ . Moreover, by \cite[Lemma 3.3.6]{Cand:Notes1}, we have that the structure groups of $(q\co\pi_1)^{-1}(h)$ are $\ab_i((q\co\pi_1)^{-1}(h))=\ab_i(\ns)$ for $i\ge 2$ and $\ab_1((q\co\pi_1)^{-1}(h))=\mb{T}^n$.
\end{proof}

We will use the above characterization to prove the following main result of this section.

\begin{theorem}\label{thm:fin-rank-implies-quasitoral}
Let $k,R\in \mb{N}$ and let $\ns$ be a $k$-step \textsc{cfr} nilspace. Then there exists $b=b(\ns,k,R)>0$ such that the following holds. If there exists a finite abelian group $\ab$ of rank at most $R$ and a $b$-balanced morphism $\varphi:\mc{D}_1(\ab)\to \ns$, then $\ns$ is quasitoral.
\end{theorem}

We reduce the proof of this theorem to proving the following two results.

\begin{proposition}\label{prop:fin-rank-reduction}
Let $k,R\in \mb{N}$ and let $\ns$ be a $k$-step \textsc{cfr} nilspace such that $\ns_1$ is connected. Then there exists $b=b(\ns,k,R)>0$ such that the following holds. If there exists a finite abelian group $\ab$ of rank at most $R$ and a $b$-balanced morphism $\varphi:\mc{D}_1(\ab)\to \ns$, then $\ns$ is toral.
\end{proposition}

\begin{proposition}\label{prop:bal-implies-restric-bal}
Let $k,R\in \mb{N}$, let $b'\in \mb{R}_{>0}$, and let $\ns$ be a $k$-step \textsc{cfr} nilspace. Then there exists $b=b(\ns,k,R,b')>0$ such that the following holds. Let $\pi_1:\ns\to\ns_1\cong \mb{T}^n\times H$, where $H$ is a finite abelian group and $n\in \mb{N}$, and let $q:\ns_1\to H$ be the projection homomorphism. If there exists a finite abelian group $\ab$ of rank at most $R$ and a $b$-balanced morphism $\varphi:\mc{D}_1(\ab)\to \ns$, then there exists a  subgroup\footnote{In particular $\ab'$ is also finite abelian of rank at most $R$.} $\ab'\le \ab$ and a $b'$-balanced morphism $\varphi':\mc{D}_1(\ab')\to (q\co \pi_1)^{-1}(0_H)$.
\end{proposition}

\begin{proof}[Proof of Theorem \ref{thm:fin-rank-implies-quasitoral} assuming Propositions \ref{prop:fin-rank-reduction} and \ref{prop:bal-implies-restric-bal}] Let us use the notation from Proposition \ref{prop:bal-implies-restric-bal}. By Lemma \ref{lem:spli-dep-on-1-factor} it suffices to prove that $(q\co\pi_1)^{-1}(0_H)$ is toral. Let $b'=b'((q\co\pi_1)^{-1}(0_H),k,R)>0$ be the parameter given by Proposition \ref{prop:fin-rank-reduction} applied to $(q\co\pi_1)^{-1}(0_H)$ and let $b=b(\ns,k,R,b')>0$ be given by Proposition \ref{prop:bal-implies-restric-bal}. 

Then, note that if there exists a finite abelian group $\ab$ of rank at most $R$ and a $b$-balanced morphism $\varphi:\mc{D}_1(\ab)\to\ns$, by Proposition \ref{prop:bal-implies-restric-bal} there exists a finite abelian group $\ab'$ of rank at most $R$ and a $b'$-balanced morphism $\varphi':\mc{D}_1(\ab')\to (q\co\pi_1)^{-1}(0_H)$. In particular, by Proposition \ref{prop:fin-rank-reduction}, we conclude that $(q\co\pi_1)^{-1}(0_H)$ is toral and the result follows.
\end{proof}

Before proceeding with the proof of Proposition \ref{prop:bal-implies-restric-bal}, let us record a technical result that we will use throughout this section.

\begin{lemma}\label{lem:cont-of-balanced}
Let $\ns,\nss$ be $k$-step \textsc{cfr} nilspaces and let $\theta:\nss\to \ns$ be a (continuous) fibration. Then for any $b>0$ there exists $b'=b'(b,\theta)>0$ such that, for any finite abelian group $\ab$, if $\varphi:\mc{D}_1(\ab)\to \nss$ is $b'$-balanced, then $\theta\co \varphi$ is $b$-balanced.
\end{lemma}

\begin{proof}
Recall that the notion of \emph{balance} (see \cite[Definition 5.1]{CSinverse}) assumes a metric on each set of probability measures $\mc{P}(\cu^n(\ns))$ and $\mc{P}(\cu^n(\nss))$ equipped with the weak topology for all $n\in \mb{N}$. Let us denote these by $d^{\ns}_n$ and $d^{\nss}_n$ respectively.

We claim that the map $\theta^{\db{n}}_*:\mc{P}(\cu^n(\nss))\to \mc{P}(\cu^n(\ns))$ that sends $\mu\mapsto \mu \co (\theta^{\db{n}})^{-1}$ is a continuous map. As both $\mc{P}(\cu^n(\ns))$ and $\mc{P}(\cu^n(\nss))$ are metric spaces, if suffices to see that, if $\mu_n\in \cu^n(\nss)\to\mu\in\cu^n(\nss)$, then for any continuous function $f:\cu^n(\ns)\to\mb{C}$ we have $\int f\;\mathrm{d}\theta^{\db{n}}_*(\mu_n)\to \int f\;\mathrm{d}\theta^{\db{n}}_*(\mu)$ as $n\to\infty$. But note that $\int f\;\mathrm{d}\theta^{\db{n}}_*(\mu_n)=\int f\co \theta^{\db{n}}\;\mathrm{d}\mu_n$ and in particular $f\co \theta^{\db{n}}$ is continuous. Thus $\int f\co \theta^{\db{n}}\;\mathrm{d}\mu_n\to \int f\co \theta^{\db{n}}\;\mathrm{d}\mu=\int f\;\mathrm{d}\theta^{\db{n}}_*(\mu)$. To conclude, note that by the Banach-Alaoglu theorem, both $\mc{P}(\cu^n(\ns))$ and $\mc{P}(\cu^n(\nss))$ are compact spaces. Thus, $\theta^{\db{n}}_*$ is uniformly continuous. In particular, for $n\le 1/b$ there exists $\delta_{n,b}>0$ such that for any $\nu,\mu\in \mc{P}(\cu^n(\nss))$ if $d^{\nss}_n(\mu,\nu)<\delta_{n,b}$, then $d^{\ns}_n(\theta^{\db{n}}_*(\mu),\theta^{\db{n}}_*(\nu))<b$. 

Finally, by \cite[Corollary 2.2.7]{Cand:Notes2} we have that $\theta^{\db{n}}_*(\mu_{\cu^n(\nss)})=\mu_{\cu^n(\nss)}\co (\theta^{\db{n}})^{-1}=\mu_{\cu^n(\ns)}$. Therefore, letting $b':=\min(1/b,\delta_{1,b},\ldots,\delta_{\lfloor 1/b\rfloor,b})$ the result follows.
\end{proof}

\begin{proof}[Proof of Proposition \ref{prop:bal-implies-restric-bal}] Let $\theta$ denote the fibration $q\co \pi_1:\ns\to \mc{D}_1(H)$ and, for every $n\in \mb{N}$, let $\theta^{\db{n}}:\cu^n(\ns)\to \cu^n(\mc{D}_1(H))$. Note that, by definition for all $n\in \mb{N}$, the cube set $\cu^n(\theta^{-1}(0_H))$ equals $(\theta^{\db{n}})^{-1}(0_H^{\db{n}})$.

We leave as an exercise for the reader to check that, for an affine homomorphism $\psi:\mc{D}_1(B)\to \mc{D}_1(A)$ where $A$ and $B$ are finite abelian groups, if $\psi$ is $b''$-balanced for some $b''=b''(A)>0$ sufficiently small, then $\psi$ is surjective. Thus, by Lemma \ref{lem:cont-of-balanced}, taking $b$ small enough depending on $H$ (which in turn depends on $\ns$) we have that $\theta\co \varphi$ is surjective.

Assuming that $b$ is small enough as stated in the previous paragraph, note that, by composing with a nilspace translation on $\ns$ permuting the fibers of $\theta$ adequately, we can assume that $\theta\co\varphi$ is a surjective homomorphism $\ab\to H$. The idea now is to restrict $\varphi$ to $\ker(\theta\co\varphi)$ (which has rank at most $R$) and consider the morphism $\varphi|_{\ker(\theta\co\varphi)}:\ker(\theta\co\varphi)\to \theta^{-1}(0_H)$. We want to prove that, for any $b'>0$, if $\varphi$ is $b$-balanced for a small enough $b=b(\ns,b')$, then $\varphi|_{\ker(\theta\co\varphi)}$ is $b'$-balanced.

To prove this we can proceed by contradiction. Suppose that for some $b'>0$ we have a sequence $(b_m)_{m\in \mb{N}}$ with $b_m\to 0$, finite abelian groups $(\ab_m)_{m\in \mb{N}}$ of rank at most $R$, and $b_m$-balanced morphisms $(\varphi_m:\mc{D}_1(\ab_m)\to\ns)_{m\in \mb{N}}$ such that, for every $m\in \mb{N}$, the restriction $\varphi_m|_{\ker(\theta\co\varphi_m)}:\ker(\theta\co\varphi_m)\to \theta^{-1}(0_H)$ is not $b'$-balanced. Then there exists some $n\le 1/b'$ and a subsequence of $m$s such that $d_n(\nu_m,\mu_{\cu^n(\theta^{-1}(0_H))})\ge b'$ where $\nu_m:=\mu_{\cu^n(\mc{D}_1(\ker(\theta\co\varphi_m)))}\co (\varphi_m^{\db{n}})^{-1}$. Abusing the notation, we shall assume that the subsequence of $m$s is the whole set of $m\in \mb{N}$. By definition of the weak topology on $\mc{P}(\cu^n(\ns))$, open sets are generated by continuous functions on $\cu^n(\ns)$, i.e., the collection of sets of the form $U_{f,W}:=\{\nu\in \mc{P}(\cu^n(\ns)): \int f\;\mathrm{d}\nu \in W\}$ where $f:\cu^n(\ns)\to \mb{C}$ is continuous and $W\subset \mb{C}$ is open forms a sub-base for the topology on $\mc{P}(\cu^n(\ns))$. Thus, as the ball of radius $b'$ around $\mu_{\cu^n(\theta^{-1}(0_H))}$ is open, there must exists a continuous function $f:\cu^n(\theta^{-1}(0_H))\to \mb{C}$ and some $\epsilon>0$ such that for all $m\in \mb{N}$
\begin{equation}\label{eq:separation-local-average}\textstyle
\big|\int f\;\mathrm{d}\nu_m-\int f\;\mathrm{d}\mu_{\cu^n(\theta^{-1}(0_H))}\big|>\epsilon>0.
\end{equation}

Now note that, as $\cu^n(\mc{D}_1(H))$ is discrete and $\theta^{\db{n}}$ is continuous, the set $\cu^n(\theta^{-1}(0_H))=(\theta^{\db{n}})^{-1}(0_H^{\db{n}})$ is closed and open inside $\cu^n(\ns)$. Thus, we can extend $f$ to a continuous function in $\cu^n(\ns)$ simply by letting $\wt{f}:\cu^n(\ns)\to \mb{C}$ be given by $x\mapsto f(x)$ if $x\in \cu^n(\theta^{-1}(0_H))$ and $x\mapsto 0$ otherwise. Hence, as $\varphi_m$ is $b_m$-balanced and $b_m\to 0$, we have that \begin{equation}\label{eq:convergence-local-average}\textstyle\int \wt{f}\;\mathrm{d} \mu_{\cu^n(\mc{D}_1(\varphi_m))}\co(\varphi_m^{\db{n}})^{-1}\to \int \wt{f}\;\mathrm{d} \mu_{\cu^n(\ns)} \text{ as }m\to\infty.
\end{equation}

However, this will contradict \eqref{eq:separation-local-average}. Indeed, for any $n \in \mb{N}$, by \cite[Lemma 2.2.10]{Cand:Notes2} combined with \cite[Lemmas 3.3.11 and 3.3.12]{Cand:Notes1}, the Haar measure on $\cu^n(\ns)$ disintegrates with respect to $\theta^{\db{n}}$. As $H$ is a finite abelian group, the Haar measure on $\cu^n(\mc{D}_1(H))$ is simply the uniform measure on the finite set $\cu^n(\mc{D}_1(H))$. Thus, $\tfrac{1}{|\cu^n(\mc{D}_1(H))|}\int \wt{f}\;\mathrm{d} \mu_{\cu^n(\ns)}=\int f\;\mathrm{d}\mu_{\cu^n(\theta^{-1}(0_H))}$. Similarly (even more easily as it is just disintegration with respect to a surjective homomorphism), we have that $\tfrac{1}{|\cu^n(\mc{D}_1(H))|}\int \wt{f}\;\mathrm{d} \mu_{\cu^n(\mc{D}_1(\varphi_m))}\co(\varphi_m^{\db{n}})^{-1}= \int f \;\mathrm{d}\nu_m$. Hence, combining these equalities with \eqref{eq:separation-local-average} and \eqref{eq:convergence-local-average} we have the desired contradiction and the result follows.\end{proof}

\noindent The rest of this section is devoted mainly to proving Proposition \ref{prop:fin-rank-reduction}. Roughly speaking, the idea of the proof is as follows. First, by induction on the step $k$ of the nilspace, we will be able to assume that $\ns$ is a coset nilspace $G/\Gamma$, and that every structure group of $\ns$ is connected except maybe the last one $\ab_k(\ns)$. Then we will proceed by contradiction: we will see that, if the last structure group of $G/\Gamma$ is not connected, then $\cu^n(G/\Gamma)$ will have \emph{many} connected components. On the other hand, for a sufficiently small $b$ we would have that the map $\varphi^{\db{n}}:\cu^n(\mc{D}_1(\ab))\to \cu^n(G/\Gamma)$ should be surjective on the set of connected components of $\cu^n(G/\Gamma)$. The contradiction will occur when we prove that the number of connected components that $\varphi^{\db{n}}$ can reach is \emph{small}.

\begin{proposition}\label{prop:lower-bnd-number-conneceted} Let $\ns$ be a $k$-step \textsc{cfr} nilspace such that $\ns_{k-1}$ is toral and the last structure group $\ab_k(\ns)$ is discrete. Then, for all $n\in \mb{N}$, the number of connected components of $\cu^n(\ns)$ is at least $|\ab_k(\ns)|^{\binom{n}{k}}$.
\end{proposition}

\begin{proof}
By \cite[Theorem 6.2]{CSinverse}, the nilspace $\ns$ is a coset nilspace $G/\Gamma$. The idea now is to adapt some ideas of \cite[proof of Theorem A.2]{CSinverse} to get an approximation of the number of connected components of $\cu^n(G/\Gamma)$.

Consider the Gray code map $\sigma_k$ on $G^{\db{k}}$, see \cite[Definition 2.2.22]{Cand:Notes1}. By \cite[Proposition 2.2.25]{Cand:Notes2}, this map restricted to $\cu^k(G)$ takes values in $G_k$. In particular, note that $\sigma_k(\cu^k(\Gamma))$ takes values in $G_k\cap \Gamma$ (abusing the notation, we will consider $\sigma_k$ as a map defined on $\cu^k(G)$). Thus, if we let $\wt{\sigma_k}:=\sigma_k\mod (G\cap \Gamma)$ recalling that $\ab_k(G/\Gamma)\cong G_k/(G_k\cap \Gamma)$, we have that $\wt{\sigma_k}:\cu^k(G)\to \ab_k(G/\Gamma)$ is a continuous function such that $\wt{\sigma_k}(\cu^k(\Gamma))=0$.

Let $J:=\cu^k(G)^0$ be the connected component of the identity in $\cu^k(G)$. This group is normal and open by standard results and thus $J\cu^k(\Gamma)$ is an open subgroup of $\cu^k(G)$. As $\ab_k(G/\Gamma)$ is discrete, it follows that $\wt{\sigma_k}(J\cu^k(\Gamma))=0$. Let $(g_i\in G_k)_{i\in|\ab_k(G/\Gamma)|}$ be a set of representatives so that $\ab_k(G/\Gamma)=\{g_i(G_k\cap \Gamma):i\in|\ab_k(G/\Gamma)|\}$ (without loss of generality, let the representative $g_1$ corresponding to $0_{\ab_k}$ be $\id\in G$) and let $g_i^F:\db{k}\to G$ be the elementary Host--Kra cube (see Definition \ref{def:hk-cubes}) with $F=\{1^k\}$. In particular, note that $\wt{\sigma_k}(g_i^FJ\cu^k(\Gamma))=g_i$ and thus the cosets $g_i^FJ\cu^k(\Gamma)$ are all different from each other. As all these cosets are invariant under quotienting by $\Gamma^{\db{k}}$, we have that $\cu^k(G/\Gamma)$ has at least $|\ab_k(G/\Gamma)|$ connected components.

To estimate now the number of connected components of $\cu^n(G/\Gamma)$, let us introduce some notation. Let $D:=\{v\in \db{n}:\sum_{i=1}^n v_i=k\}$. For each $v\in D$ and each $g\in G_k$ let $g^v\in \cu^n(G)$ be given as $g^v\sbr{w}=g$ if for all $i\in[n]$ we have $v_i\le w_i$ and $g^v\sbr{w}=\id$ otherwise. We now claim that the elements $\{\prod_{v\in D }(h_v)^v\Gamma^{\db{n}}: h_v\in \{g_1,\ldots,g_{|\ab_k(G/\Gamma)|}\}\text{ for }v\in\db{n}\}$ lie all in different connected components in $\cu^n(G/\Gamma)$.

To prove this, let us show that the element $\Gamma^{\db{n}}$ lies in a different connected component from any other element $\prod_{v\in D }(h_v)^v\Gamma^{\db{n}}$ (as the proof in general is essentially the same). Without loss of generality, assume that $h_{(1^k,0^{n-k})}\not=\id$ (recall that we have assumed that $g_1=\id$). Let $p:\cu^n(G/\Gamma)\to \cu^k(G/\Gamma)$ be the projection $\q\mapsto \q\co\phi$ where $\phi:\db{k}\to\db{n}$ is the discrete-cube morphism  $w\mapsto (w,0^{n-k})$. As this map $p$ is continuous, if $\prod_{v\in D }(h_v)^v\Gamma^{\db{n}}$ lay in the same component as $\Gamma^{\db{n}}$ then  $p(\prod_{v\in D }(h_v)^v\Gamma^{\db{n}})=h_{(1^k,0^{n-k})}^F\Gamma^{\db{k}}$ would lie in the same component as $p(\Gamma^{\db{n}})=\Gamma^{\db{k}}$. But we know that these two elements lie in different connected components and thus the result follows.

Therefore, in $\cu^n(\ns)=\cu^n(G/\Gamma)$ we have at least $|\ab_k(G/\Gamma)|^{|D|}=|\ab_k(G/\Gamma)|^{\binom{n}{k}}$ different connected components.\end{proof}

Proposition \ref{prop:lower-bnd-number-conneceted} gives us a \emph{lower} bound on the number of connected components of $\cu^n(\ns)$. Let us now prove that if a morphism $\varphi:\mc{D}_1(\ab)\to \ns$ is sufficiently balanced, then for each connected component $C\subset \cu^n(\ns)$ there is some cube $\q\in\cu^n(\mc{D}_1(\ab))$ such that $\varphi\co\q\in C$, which will yield an \emph{upper} bound on the number of components. More precisely, we shall prove the following result.

\begin{lemma}\label{lem:bal-implies-surj}
Let $\ns$ be a $k$-step \textsc{cfr} nilspace such that $\ns_{k-1}$ is toral and let $n\in\mb{N}$. Then there exists $b=b(\ns,n)>0$ such that if $\ab$ is a finite abelian group and $\varphi:\mc{D}_1(\ab)\to\ns$ is $b$-balanced then the following holds. For any connected component $C\subset \cu^n(\ns)$ there exists $\q\in\cu^n(\mc{D}_1(\ab))$ such that $\varphi\co\q\in C$.
\end{lemma}

\begin{proof}
By \cite[Lemma A.3]{CSinverse}, all connected components of $\cu^n(\ns)$ have equal Haar measure. In particular, there are finitely many such components, say $C_1,\ldots,C_\ell$, and $\mu_{\cu^n(\ns)}(C_i)=1/\ell$ for all $i\in[\ell]$. As those components must be open and closed sets, the indicator function $1_{C_i}:\cu^n(\ns)\to \mb{C}$ is a continuous function. Let $b<1/n$. If $\varphi$ is $b$-balanced, then $d_n(\mu_{\cu^n(\mc{D}_1(\ab))}\co (\varphi^{\db{n}})^{-1},\mu_{\cu^n(\ns)})<b$ where $d_n$ is a distance in the space of Borel probability measures on $\cu^n(\ns)$. Thus, for each $i\in[\ell]$ there exists $b_i>0$ such that if $d_n(\mu_{\cu^n(\mc{D}_1(\ab))}\co (\varphi^{\db{n}})^{-1},\mu_{\cu^n(\ns)})<b_i$ then $|\int 1_{C_i}\;\mathrm{d}\mu_{\cu^n(\mc{D}_1(\ab))}\co (\varphi^{\db{n}})^{-1}-\int 1_{C_i}\;\mathrm{d}\mu_{\cu^n(\ns)}|=|\int 1_{C_i}\;\mathrm{d}\mu_{\cu^n(\mc{D}_1(\ab))}\co (\varphi^{\db{n}})^{-1}-1/\ell|<1/(2\ell)$. Letting $b:=\min(1/n,b_1,\ldots,b_\ell)>0$ the result follows.
\end{proof}

The last ingredient that we shall need is a result telling us that, if $\ab$ has rank at most $R$, then there is a useful \emph{upper} bound on the number of connected components of $\cu^n(\ns)$ containing elements of the form $\varphi\co\q$ for $\q\in \cu^n(\mc{D}_1(\ab))$. First, let us prove a technical result.

\begin{lemma}\label{lem:fact-of-morphisms}
Let $\ns$ be a $k$-step \textsc{cfr} coset nilspace such that $\ns_{k-1}$ is toral, let $R\in \mb{N}$, and let $\varphi:\mc{D}_1(\mb{Z}^r)\to \ns$ be a morphism for some $r\in [R]$. Then there exists a nilmanifold $G/\Gamma$ such that $\ns\cong G/\Gamma$ and, letting $G^0$ be the connected component of $G$, the following holds. There exists $g\in \hom(\mc{D}_1(\mb{Z}^r),G^0)$ and $f\in \hom(\mc{D}_1(\mb{Z}^r),\mc{D}_k(\ab_k(\ns)))$ such that $\varphi(\cdot)= g(\cdot)\Gamma+f(\cdot)$.  
\end{lemma}

\begin{proof}
By \cite[Theorem A.1]{CSinverse}, the nilspace $\ns$ is isomorphic to $G/\Gamma$ where $G=\tran(\ns)$ with filtration $G_\bullet=(\tran_i(\ns))_{i\in [k]}$ and, for a fixed $x_0\in \ns$, we have that $\Gamma=\stab_G(x_0)$. The map $\pi_{k-1}\co \varphi:\mc{D}_1(\mb{Z}^r)\to \ns_{k-1}$ takes values in a toral nilspace. By \cite[Theorem 2.9.17]{Cand:Notes2}, the nilspace $\ns_{k-1}\cong \tran(\ns_{k-1})^0/\Gamma'$ where $\Gamma'=\stab_{\tran(\ns_{k-1})^0}(x_0)$. In particular, there exists $g'\in \hom(\mc{D}_1(\mb{Z}^r),\tran(\ns_{k-1})^0)$ such that $\pi_{k-1}\co\varphi(\cdot)=g'(\cdot)\Gamma'$. In particular, we have that $g'(v_1,\ldots,v_r)=\prod_{\underline{w}\in\mb{Z}_{\ge 0}:w_1+\cdots+w_r\le k} h_{\underline{w}}^{\binom{\underline{v}}{\underline{w}}}$ where $\binom{\underline{v}}{\underline{w}}=\binom{v_1}{w_1}\cdots\binom{v_r}{w_r}$ and $h_{\underline{w}}\in \tran_{w_1+\cdots+w_r}(\ns_{k-1})^0$. By \cite[Proposition 2.9.20]{Cand:Notes2}, for each $h_{\underline{w}}$ there exists $h'_{\underline{w}}\in \tran_{w_1+\cdots+w_r}(\ns)^0$ such that $\wh{\pi_{k-1}}(h'_{\underline{w}})=h_{\underline{w}}$.\footnote{This means that for any $x\in\ns$, we have $\pi_{k-1}\co h'_{\underline{w}}(x)=h_{\underline{w}}\co \pi_{k-1}(x)$.} Let us define $g(v_1,\ldots,v_r):=\prod_{\underline{w}\in\mb{Z}_{\ge 0}:w_1+\cdots+w_r\le k} (h'_{\underline{w}})^{\binom{\underline{v}}{\underline{w}}}$.

Clearly, we then have that $g(\cdot)\Gamma$ is a morphism in $\hom(\mc{D}_1(\mb{Z}^r),\ns)$ such that $\pi_{k-1}\co g(\cdot)\Gamma = \pi_{k-1}\co \varphi$. Therefore $\varphi(\cdot)=g(\cdot)\Gamma+f(\cdot)$ where $f\in\hom(\mc{D}_1(\mb{Z}^r),\mc{D}_k(\ab_k(\ns)))$, and the result follows.\end{proof}

\begin{lemma}\label{lem:few-reachable-comp}
Let $\ns$ be a $k$-step \textsc{cfr} coset nilspace such that $\ns_{k-1}$ is toral and $\ab_k(\ns)$ is discrete, let $R\in \mb{N}$, and let $\varphi:\mc{D}_1(\mb{Z}^r)\to \ns$ be a morphism for some $r\in [R]$. Then, for every $n\in \mb{N}$, the image of $\cu^n(\mc{D}_1(\mb{Z}^r))$ under the map $\varphi^{\db{n}}$ is included in a union of at most $(|\ab_k(\ns)|k!)^{R(n+1)}$ connected components of $\cu^n(\ns)$.
\end{lemma}

\begin{proof}
By Lemma \ref{lem:fact-of-morphisms}, we can assume that $\ns\cong G/\Gamma$ and that $\varphi=g\Gamma+f$ where $g\in \hom(\mc{D}_1(\mb{Z}^r),G^0)$ and $f\in \hom(\mc{D}_1(\mb{Z}^r),\mc{D}_k(\ab_k(\ns)))$. The key observation is that the image of $g$ lies entirely within $\cu^n(G^0)$. Thus, the image of $g^{\db{n}}\Gamma^{\db{n}}$ lies within the image under $\cu^n(G)\to \cu^n(G/\Gamma)$ of the connected component $\cu^n(G)^0$ which is thus connected. Therefore, the number of connected components that $\varphi^{\db{n}}$ can reach is at most given by the possible different values of the set $f^{\db{n}}(\cu^n(\mc{D}_1(\mb{Z}^r)))\subset \cu^n(\mc{D}_k(\ab_k(\ns)))$. But note that $f:\mc{D}_1(\mb{Z}^r)\to \mc{D}_k(\ab_k(\ns))$ is a morphism and thus, it has a Taylor expansion $f(v_1,\ldots,v_r)=\sum_{\underline{w}\in \mb{Z}^r_{\ge 0}:w_1+\cdots+w_r\le k} a_{\underline{w}}\binom{\underline{v}}{\underline{w}}$ where $a_{\underline{w}}\in \ab_k(\ns)$ and $\binom{\underline{v}}{\underline{w}}:=\binom{v_1}{w_1}\cdots\binom{v_r}{w_r}$. Thus, it is easy to see that $f$ is $|\ab_k(\ns)|k!$-periodic. In particular\footnote{Here we use, as elsewhere in this paper, the notation $\mb{Z}_n$ for the cyclic group of integers modulo $n$.} $|f^{\db{n}}(\cu^n(\mc{D}_1(\mb{Z}^r)))|\le |\cu^n(\mb{Z}_{|\ab_k(\ns)|k!}^r)|\le (|\ab_k(\ns)|k!)^{r(n+1)}\le (|\ab_k(\ns)|k!)^{R(n+1)}$, and the result follows.
\end{proof}

\begin{proof}[Proof of Proposition \ref{prop:fin-rank-reduction}]
We prove the result by induction on $k$. Note that the case $k=1$ is true by hypothesis and thus we shall assume that $k\ge 2$, that $\ns$ is $k$-step, and that Proposition \ref{prop:fin-rank-reduction} holds for step up to $k-1$. Let $b_1'=b_1'(\ns_{k-1},k-1,R)>0$ be the parameter given by Proposition \ref{prop:fin-rank-reduction} applied to $\ns_{k-1}$. Let $b_1=b_1(b_1',\pi_{k-1})$ be the parameter given by Lemma \ref{lem:cont-of-balanced} applied to the projection map $\pi_{k-1}:\ns\to\ns_{k-1}$ and $b_1'$. Letting $b<b_1$, we may assume that $\ns_{k-1}$ is toral. 

Let now $\ns'$ be the quotient of $\ns$ by the toral part of $\ab_k(\ns)$, see \cite[Proposition A.19]{CGSS-p-hom}. Hence, we have a fibration $\phi:\ns\to\ns'$ and let us denote $\phi\co\varphi$ by $\varphi'$. If $\ab_k(\ns')$ is trivial then we are done. Otherwise, let $n=n(\ns',k,R)>0$ be large enough such that \begin{equation}\label{eq:bnd-of-conneceted-comp}|\ab_k(\ns')|^{\binom{n}{k}}>(|\ab_k(\ns')|k!)^{R(n+1)}.\end{equation} By Lemma \ref{lem:bal-implies-surj} combined with proposition \ref{prop:lower-bnd-number-conneceted}, for some $b_2'=b_2'(n,\ns',R)>0$ small enough, we have that $(\varphi')^{\db{n}}(\cu^n(\mc{D}_1(\ab)))$ lie in at least $|\ab_k(\ns')|^{\binom{n}{k}}$ different connected components of $\cu^n(\ns')$. On the other hand, by Lemma \ref{lem:few-reachable-comp} (and noting that we may take a surjective homomorphism $\mb{Z}^R\to \ab$ as the rank of $\ab$ is at most $R$), we have that $(\varphi')^{\db{n}}(\cu^n(\mc{D}_1(\ab)))$ lie in at most $(|\ab_k(\ns)|k!)^{R(n+1)}$ connected components of $\cu^n(\ns')$, and this contradicts \eqref{eq:bnd-of-conneceted-comp}. Hence, if $\varphi'$ is sufficiently balanced we have that $\ab_k(\ns')$ must be the trivial group $\{0\}$.

Now let $b_2=b_2(b_2',\phi)>0$ be as given by Lemma \ref{lem:cont-of-balanced} applied to $\phi$ and $b_2'$. Finally, letting $b:=\min(b_1,b_2)/2$, the result follows.
\end{proof}

To finish this section we record the following consequence which completes the first main step in our strategy.

\begin{proposition}\label{prop:mainstep1}
For every $k,R\in \mb{N}$ and $\delta>0$, there exists $C_0>0$ and $\varepsilon_0>0$ such that the following holds. Let $\ab$ be a finite abelian group of rank at most $R$, and let $f:\ab\to\mb{C}$ be a 1-bounded function satisfying $\|f\|_{U^{k+1}}\geq \delta$. Then there exists a subgroup $\ab_0\leq \ab$ of index at most $C_0$, a filtered nilmanifold $G_0/\Gamma_0$ of degree $k$ where the filtration on $G_0$ consists of connected and simply-connected Lie groups, such that the associated toral nilspace $\nss=G_0/\Gamma_0$ is of complexity at most $C_0$, a 1-bounded Lipschitz function $F_0:\nss\to \mb{C}$ of Lipschitz norm at most $C_0$, a morphism $\varphi_0:\mc{D}_1(\ab_0)\to\nss$ and some element $t_0\in \ab$ such that 
\begin{equation}\label{eq:mainstep1}
|\mb{E}_{y\in \ab_0}f(t_0+y) F_0(\varphi_0(y))|\geq \varepsilon_0.
\end{equation}
\end{proposition}

\begin{proof}
Fix a function $b_{k,R}:\mb{R}_{>0}\to\mb{R}_{>0}$ such that, for every $k$-step \textsc{cfr} nilspace $\ns$ of complexity at most $m$, we have that $b(m)$ is less than the parameter $b(\ns,k,R)$ required to obtain the conclusion of Theorem \ref{thm:fin-rank-implies-quasitoral}. By Theorem 5.2 from \cite{CSinverse} applied with this function $b$ and $k,\delta$, there is $M>0$ such that, for the given function $f$, there exists $m\leq M$ and a $b(m)$-balanced 1-bounded nilspace polynomial $F\co\varphi$ on $\ab$ of complexity at most $m$ such that $\langle f,F\co\varphi\rangle\geq \varepsilon_0:=\delta^{2^{k+1}}/2$. In particular the morphism $\varphi:\ab\to\ns$ is sufficiently balanced to that, by Theorem \ref{thm:fin-rank-implies-quasitoral}, the nilspace $\ns$ must be quasitoral.

Let $H$ be the finite abelian group such that $\ns_1\cong H\times \mb{T}^n$ and note that  $q\co\pi_1\co\varphi$ is a surjective affine homomorphism $\ab\to H$, i.e.\ a map $z\mapsto \theta(z)+h_0$, where $\theta:\ab\to H$ is a surjective homomorphism. 
By Lemma \ref{lem:spli-dep-on-1-factor}, we have a partition
\begin{equation}
   \ns=\bigsqcup_{h\in H} \nss_h 
\end{equation}
where for every $h\in H$ the nilspace $\nss_h:=(q\co\pi_1)^{-1}(h)$ is toral. Let $\ab_0:=\ker(\theta)$, a subgroup of $\ab$ of index $|H|$, and note that for every $h\in H$ there is a coset $t_h+\ab_0$ such that $\varphi^{-1}(\nss_h)=t_h+\ab_0$. We therefore have $\varepsilon_0\leq \mb{E}_{x\in \ab}f(x) F(\varphi(x)) = \mb{E}_{h\in H} \mb{E}_{y\in \ab'}f(t_h+y) F(\varphi(t_h+y))$, so there exists $h^*\in H$ such that, relabeling $t_{h^*}$ as $t'$, we have $\varepsilon_0\leq |\mb{E}_{y\in \ab'}f(t'+y) F(\varphi(t'+y))|$.

Note that since $\nss:=\nss_{h^*}$ is a toral nilspace, it is a filtered nilmanifold $(G_0/\Gamma_0,{G_0}_\bullet)$ of degree $k$, where the groups in ${G_0}_\bullet$ can be taken to be connected and simply-connected by Remark \ref{rem:toral&metric}, and the map $y\mapsto \varphi(t'+y)$ is a polynomial map $\ab_0\to G_0/\Gamma_0$. 
\end{proof}

\section{Extending nilsequences}\label{sec:extend}

\noindent It is a basic result of classical Fourier analysis that Fourier characters on closed subgroups of a compact abelian group can be extended to the full group (see \cite[Theorem 2.1.4]{RudinFAOG}). In this section prove a generalization of this in the case of nilsequences on finite abelian groups.

The goal of the argument is to extend the nilsequence $F_0\co\varphi_0$ from the subgroup $\ab_0\leq\ab$ given in \eqref{eq:mainstep1} to all of $\ab$. To this end, we begin with the following lemma which indicates what types of elementary steps can arise in the extension process.

\begin{lemma}\label{lem:ladder}
Let $\ab$ be a finite abelian group and let $\ab_0$ be a subgroup of $\ab$. Then there is a sequence of subgroups
\begin{equation}\label{eq:ladder}
\ab_0\leq \ab_1 \leq \cdots \leq \ab_t=\ab
\end{equation}
where $t\leq \log_2(|\ab|/|\ab_0|)$ and for every $i$ the group $\ab_i$ is an extension of the group $\ab_{i-1}$ of precisely one of the following two types:
\begin{enumerate}
\item (split case) there is a prime $p$ such that $\ab_i\cong \mb{Z}_p\oplus \ab_{i-1}$.
\item (non-split case) there is a prime $p$, a positive integer $d$ and a subgroup $K\leq \ab_{i-1}$ such that $\ab_{i-1}\cong (p\cdot \mb{Z}_{pd}) \oplus K \leq  \mb{Z}_{pd}\oplus K \cong \ab_i$.
\end{enumerate}
\end{lemma}

\begin{proof}
By the fundamental theorem of finite abelian groups, we have
\begin{equation}
\ab_0 \cong \bigoplus_{i\in [s]} \mb{Z}_{p_i^{e_{i,1}}} \oplus \mb{Z}_{p_i^{e_{i,2}}} \oplus \cdots\oplus \mb{Z}_{p_i^{e_{i,\ell_i}}}
\end{equation}
where the primes $p_1,\ldots,p_s$ satisfy $p_i<p_{i+1}$ for every $
i\in [s-1]$ and the positive integers $e_{i,j}$ satisfy $e_{i,j}\leq e_{i,j+1}$.

As every prime $p_i$ with $i\in [s]$ divides $|\ab|$, by the fundamental theorem applied to $\ab$, for some additional primes $p_{s+1}<\cdots<p_r$ all greater than $p_s$ we have a similar decomposition
\begin{equation}
\ab \cong \bigoplus_{i\in [r]} \mb{Z}_{p_i^{e'_{i,1}}} \oplus \mb{Z}_{p_i^{e'_{i,2}}} \oplus \cdots\oplus \mb{Z}_{p_i^{e'_{i,\ell_i'}}}.
\end{equation}
It is then clear that we can proceed from $\ab_0$ to $\ab$ by a sequence of extensions as claimed.

Note that $t$ is just the number of prime factors (counting repetitions) of $|\ab|/|\ab_0|$, which is at most $\log_2(|\ab|/|\ab_0|)$.
\end{proof}
\noindent The idea of the argument is to show that the extension of a nilsequence on $\ab_{i-1}$ to $\ab_i$ can be obtained in each of the above two cases, and then iterate from $\ab_0$ up to $\ab$. 

We begin with the split case, in which the extension is trivial.

\begin{lemma}[Extending in the split case]\label{lem:splitcase}
Suppose that for some finite abelian group $\ab$ and some prime $p$ we have a morphism  $g$ from $\mc{D}_1(\ab)$ to a nilspace $\ns$. Then the map $\wt{g}(z,x):=g(x)$ is a morphism from $\mb{Z}_p\oplus\ab$ to $\ns$.
\end{lemma}
\begin{proof}
This is immediate since the projection $\mb{Z}_p\oplus\ab\to \ab$, $(z,x)\mapsto x$ is a homomorphism and therefore its composition with $g$ is indeed a nilspace morphism.
\end{proof}

\noindent In this split case we thus directly extend the polynomial map underlying the given nilsequence on $\ab_{i-1}$. In particular, the underlying nilmanifold $G/\Gamma$ remains the same in the extended nilsequence as in the original one. The non-split case, to which we now turn, is much less trivial. In fact, it is necessary in this case to modify the nilmanifold underlying the original nilsequence on $\ab_{i-1}$, in order to be able to extend the nilsequence to $\ab_i$. Let us pause to justify this with the following example. (Note that an almost-identical example appeared in \cite[Remark 4.3]{CGSS-bndtor}; here we revisit this and are able to give a much more elementary treatment.)

\begin{example}[A non-extendable polynomial map]\label{ex:nonext}
Fix any prime $p$ and let $\ab=\mb{Z}_{p^2}\times \mb{Z}_p$ and $\ab_0=(p\mb{Z}_{p^2})\times \mb{Z}_p\leq \ab$. Consider the morphism $g\in \hom(\mc{D}_1(\ab_0),\mc{D}_2(\mb{T}))$ defined by $g(px,y)=\frac{xy}{p}$, where the point $\frac{xy}{p}\in \mb{T}$ is obtained by taking any integer representatives of $x$ (mod $p^2$) and $y$ (mod $p$) and then taking $\frac{xy}{p}\mod 1$ (note that this is well-defined and indeed a quadratic map, i.e.\ a morphism into $\mc{D}_2(\mb{T})$ as claimed). We claim that there is no morphism $\wt{g}:\mc{D}_1(\ab)\to\mc{D}_2(\mb{T})$ such that $\wt{g}|_{\ab_0}=g$. Indeed, suppose for a contradiction that such a map $\wt{g}$ exists. Then, composing it with the natural surjective homomorphism $\mb{Z}^2\to \ab$ (i.e.\ reduction mod $p^2$ in the first coordinate and mod $p$ in the second), we obtain a morphism $h:\mc{D}_1(\mb{Z}^2)\to \mc{D}_2(\mb{T})$ which is $p^2$-periodic in the first coordinate, and $p$-periodic in the second coordinate. Moreover, from the assumption that $\wt{g}(px,y)=g(px,y)=\frac{xy}{p}\mod 1$, we deduce that for all integers $x,y$ we have $h(px,y)= \frac{xy}{p}$. We thus have the following system of constraints:
\begin{align}
 \forall x,y\in \mb{Z},  & \quad h(x+p^2,y)=h(x,y) \label{eq:circ1} \\ 
 \forall x,y\in \mb{Z}, & \quad h(x,y+p)=h(x,y) \label{eq:circ2} \\
 \forall x,y\in \mb{Z}, & \quad h(px,y)=\frac{xy}{p}.\label{eq:circ3}
\end{align}
On the other hand, since $h$ is a quadratic map into $\mb{T}$, it has a Taylor expansion of the form $h(x,y)=a_1+a_2x+a_3y+a_4\binom{x}{2}+a_5\binom{y}{2}+a_6xy$ for some $a_i\in \mb{T}$. Now note that from equation \eqref{eq:circ3} evaluated successively at $(x,y)=(0,0),(0,1),(0,2)$, we deduce that $a_1=0$, then $a_3=0$, and then $a_5=0$. Hence  $h(x,y)=a_2x+a_4\binom{x}{2}+a_6xy$. Now, applying this in \eqref{eq:circ2} we deduce that $a_6xp=h(x,y+p)-h(x,y)=0$ for all $x$, whence $a_6p=0$. Then, using this back in \eqref{eq:circ3} we obtain that $a_2px+a_4\binom{px}{2}=\frac{xy}{p}$, which, setting $x=1$, gives us the absurd claim that the constant $a_2p+a_4\binom{p}{2}$ equals $\frac{y}{p}$ for all $y$. This proves that the extension $\wt{g}$ cannot exist.\footnote{Note that we did not use equation \eqref{eq:circ1} to derive the contradiction, only \eqref{eq:circ2} and \eqref{eq:circ3}. This is not concerning because it can be checked that equation \eqref{eq:circ1} is actually redundant; more precisely, this equation can be deduced from equation \eqref{eq:circ3} by evaluating the latter further at $(1,0)$ and $(2,0)$ and deducing from this that $a_2$, $a_4$ and $a_6$ must be such that $h$ is forced to be $p^2$-periodic in $x$.}
\end{example}

\noindent The above example shows that, to obtain the desired nilsequence extension in the non-split case, we cannot in general conserve the  nilmanifold underlying the nilsequence. We shall deal with this difficulty by proving and applying a result which enables us to ``lift" a multivariable polynomial nilsequence to a linear multivariable nilsequence, Proposition \ref{prop:exte-non-split} below. In the univariable case (i.e.\ for nilsequences defined on $\mb{Z}$, as opposed to the multivariable case of $\mb{Z}^r$ with $r>1$), such a lifting result is known from work of Green, Tao and Ziegler, namely Proposition C.2 in \cite{GTZ} (see also \cite[Ch.\ 14, \S 2]{HKbook}).

Given a filtration $G_\bullet$ and $j\in\mb{Z}_{\ge0}$, we denote by $G_\bullet^{+j}$ the shifted filtration defined by $G_i^{+j}:=G_{i+j}$ for $i\geq 0$. Recall that we denote by $\hom(\ns,\nss)$ the set of morphisms from a nilspace $\ns$ to a nilspace $\nss$. When $\nss$ is the group nilspace associated with a filtered group $(G,G_\bullet)$, we abbreviate the notation to $\hom(\ns,G_\bullet)$.

In the proof of Proposition \ref{prop:exte-non-split} below, we will use the following semidirect product construction to define the appropriate nilmanifold for extending nilsequences. 

\begin{lemma}\label{lem:generalized-shift}
Let $k\in \mb{Z}_{\geq 0}, r\in \mb{N}$, let $(G,G_\bullet)$ be a prefiltered Lie group of degree $k$ where for all $i\in [k]$, the group $G_i$ is connected and simply-connected. Then $\mf{G}:=\hom(\mc{D}_1(\mb{Z}^r),G_\bullet)$ is a connected and simply-connected Lie group which is isomorphic to $\hom(\mc{D}_1(\mb{R}^r),G_\bullet)$, with isomorphism given by $g\mapsto g|_{\mb{Z}^r}$. In particular, letting $\vartheta:\mb{R}^r\to \aut(\mf{G})$ be the smooth homomorphism defined by $\vartheta_{\underline{x}}(g)=g(\cdot+{\underline{x}})$, we have that $\mf{G}\rtimes_\vartheta \mb{R}^r$ is a well-defined connected and simply-connected Lie group.
\end{lemma}

The proof will use the following Taylor expansion.

\begin{lemma}\label{lem:RealTaylor}
Let $(G,G_\bullet)$ be a prefiltered Lie group of degree $k$ where for all $i\in [k]$, the Lie group $G_i$ is connected and simply-connected. Then any map $g\in \hom(\mc{D}_1(\mb{R}^r),G_\bullet)$ has a Taylor expansion of the form \begin{equation}\label{eq:taylor-cont}g(\underline{x})=\prod_{\underline{j}\in J}h_{\underline{j}}^{\binom{\underline{x}}{\underline{j}}}
\end{equation}
for all $\underline{x}\in \mb{R}^r$, where $J:=\{\underline{n}\in \mb{Z}^r_{\ge0}: |n|:= n_1+\cdots+n_r\le k\}$, $\binom{\underline{x}}{\underline{j}}:=\binom{x_1}{j_1}\cdots\binom{x_r}{j_r}$, and for $\underline{j}\in J$ we have $h_{\underline{j}}\in G_{|j|}$.
\end{lemma}

\begin{proof}
Note that the powers $h_{\underline{j}}^{\binom{\underline{x}}{\underline{j}}}$ in \eqref{eq:taylor-cont} are well-defined functions of $x$, since the connectedness and simple-connectedness enables us\footnote{In a more general setting (without simple connectedness), an expression as in the right side of \eqref{eq:taylor-cont} can be defined by choosing (non-uniquely) 1-parameter subgroups $h_{\underline{j}}$, but we do not need to go in this direction in this paper.} to define $h_{\underline{j}}^t:=\exp(t\log h_{\underline{j}})$ for any $t\in \mb{R}$. 

First, note that letting $i\in[k]$ and $h^t$ be a 1-parameter subgroup in $G_i$, we have that $h^{\binom{\underline{x}}{\underline{j}}}$ is a continuous morphism in $\hom(\mc{D}_1(\mb{R}^r),G_\bullet)$. By \cite[Theorem 2.2.14]{Cand:Notes1} and the fact that $\poly(\mc{D}_1(\mb{R}^r),G_\bullet)$ is a group under point-wise multiplication, we have that any expression of the form \eqref{eq:taylor-cont} is an element of $\hom(\mc{D}_1(\mb{R}^r),G_\bullet)$.

To prove now that any $g\in \hom(\mc{D}_1(\mb{R}^r),G_\bullet)$ has one such expression, we proceed by induction on $k$. The base case $k=0$ is trivial. Assume then that the claim holds for degree up to $k-1$, and let $g\in \hom(\mc{D}_1(\mb{R}^r),G_\bullet)$. Then $gG_k\in \hom(\mc{D}_1(\mb{R}^r),G_\bullet/G_k)$ and by induction there exists 1-parameter subgroups $\wt{h}_{\underline{j}}^t$ in $G_{|j|}/G_k$ such that $gG_k = \prod_{\underline{j}\in \wt{J}}\wt{h}_{\underline{j}}^{\binom{\underline{x}}{\underline{j}}}$ where $\wt{J}:=\{\underline{n}\in \mb{Z}^r_{\ge0}:|n|\le k-1\}$. By \cite[Lemma 4.19]{H&M-ProLie}, for every $\underline{j}\in \wt{J}$ there is a 1-parameter subgroup  $h_{\underline{j}}^t$ such that $h_{\underline{j}}^tG_k=\wt{h}_{\underline{j}}^t$. Thus, letting $g':=\prod_{\underline{j}\in \wt{J}}h_{\underline{j}}^{\binom{\underline{x}}{\underline{j}}}$ we have that $g'\in \hom(\mc{D}_1(\mb{R}^r),G_\bullet)$ and that $g'G_k=gG_k$. Hence, the difference $g-g'$ is a map in $\hom(\mc{D}_1(\mb{R}^r),\mc{D}_k(G_k))$. From the assumptions it follows that $G_k=\mb{R}^d$ for some $d\geq 0$, and thus by \cite[Lemma 3.6]{CGSS-doucos} the map $g-g'$ has a Taylor expansion which we can add to that of $g'$ to complete the Taylor expansion of $g=g'+(g-g')$.    
\end{proof}

\begin{proof}[Proof of Lemma \ref{lem:generalized-shift}]
It follows from the assumptions that $G_k$ is a connected and simply-connected abelian Lie group, and is therefore isomorphic to $\mb{R}^d$ for some $d\in \mb{Z}_{\geq 0}$. Moreover, as the exponential map for $G$ is a diffeomorphism  (i.e.\ a $C^\infty$ bijection whose inverse is also $C^\infty$), it follows that $G/G_k$ is also connected and simply-connected. Therefore, for each $i\in[k]$ we have that $(G/G_i,G_\bullet/G_i)$ is a degree-$(i-1)$ connected and simply-connected filtered Lie group and every group in this filtration is connected.

Let $\psi:\hom(\mc{D}_1(\mb{R}^r),G_\bullet)\to \hom(\mc{D}_1(\mb{Z}^r),G_\bullet)$ be the map $g\mapsto g|_{\mb{Z}^r}$. This is clearly a continuous homomorphism, and we claim that it is bijective. 

To see that $\ker(\psi)$ is trivial, let $g\in \hom(\mc{D}_1(\mb{R}^r),G_\bullet)$ be such that $g|_{\mb{Z}^r}=\id$, and let us prove by induction on the degree $k$ that then $g=\id$. For $k=0$ the map $g$ must be a constant (as $G_1$ is trivial) and the claim is clear. Suppose then that $k\geq 1$ and the claim holds for $k-1$. Then $gG_k\in \hom(\mc{D}_1(\mb{R}^r),G_\bullet/G_k)$ is a polynomial map such that $(gG_k)|_{\mb{Z}^r}=G_k$. By induction we know that $(gG_k)=G_k$ so in particular $g\in \hom(\mc{D}_1(\mb{R}^r),\mc{D}_k(G_k))$. Since $G_k=\mb{R}^d$, by \cite[Lemma 3.6]{CGSS-doucos} it  follows that if $g|_{\mb{Z}^r}=\id$ then then all Taylor coefficients of $g$ must be 0, so $g=\id$.

To see that $\psi$ is surjective, note that any element $f\in\mf{G}$ has a Taylor expansion $f=\prod_{\underline{j}\in \wt{J}}h_{\underline{j}}^{\binom{\underline{x}}{\underline{j}}}$ where $h_{\underline{j}}\in G_{|\underline{j}|}$ by \cite[Lemma B.9]{GTZ}. As all groups $G_i$ for $i\in[k]$ are nilpotent and connected, there exists a (unique) 1-parameter subgroups $h_{\underline{j}}^t$ such that $h_{\underline{j}}^1=h_{\underline{j}}$. But then letting $g:=\prod_{\underline{j}\in \wt{J}}h_{\underline{j}}^{\binom{\underline{x}}{\underline{j}}}$ (which is an element of $\hom(\mc{D}_1(\mb{R}^r),G_\bullet)$ by Lemma \ref{lem:RealTaylor}), where abusing the notation now $h_{\underline{j}}^t$ are the aforementioned 1-parameter subgroups, we clearly have that $g|_{\mb{Z}^r}=f$.

To see that $\hom(\mc{D}_1(\mb{Z}^r),G_\bullet)$ is a Lie group note that the uniqueness of the Taylor expansion induces a diffeomorphism from $\hom(\mc{D}_1(\mb{Z}^r),G_\bullet)$ to $ \prod_{\underline{j}\in J}G_{|\underline{j}|}$ that sends each $g\in \hom(\mc{D}_1(\mb{Z}^r), G_\bullet)$ to the $J$-tuple of its Taylor coefficients. Indeed, the map that sends an element of $ \prod_{\underline{j}\in J}G_{|\underline{j}|}$ to the corresponding polynomial is clearly $C^\infty$, and the fact that Taylor coefficients of any $g\in \hom(\mc{D}_1(\mb{Z}^r),G_\bullet)$ can be computed as products of evaluations of $g$ in the set $J$ makes the map $C^\infty$ as well. Thus, we may identify the group $\hom(\mc{D}_1(\mb{Z}^r), G_\bullet)$ with $ \prod_{\underline{j}\in J}G_{|\underline{j}|}$ with a group operation inherited from that of $\hom(\mc{D}_1(\mb{Z}^r), G_\bullet)$. This proves that $\hom(\mc{D}_1(\mb{Z}^r), G_\bullet)$ is a Lie group. Using that all terms in the filtration $G_\bullet$ are connected and simply-connected, it follows from the identification of $\hom(\mc{D}_1(\mb{Z}^r),G_\bullet)$ with $ \prod_{\underline{j}\in J}G_{|\underline{j}|}$ that $\mf{G}$ is also connected and simply-connected.

To complete the proof, note that $\vartheta:\mb{R}^r\to \aut(\mf{G})$ is now simply the shift operator defined on $\hom(\mc{D}_1(\mb{R}^r),G_\bullet)$ which is differentiable. To see this, note that we can use the identification of $\hom(\mc{D}_1(\mb{Z}^r),G_\bullet)$ with $\prod_{\underline{j}\in J}G_{|\underline{j}|}$. Then $\vartheta_{\underline{x}}$ corresponds to sending any element $(g_{\underline{j}})_{\underline{j}\in J}\in \prod_{\underline{j}\in J}G_{|\underline{j}|}$ to the polynomial $\underline{z}\mapsto \prod_{\underline{j}\in J}g_{\underline{j}}^{\binom{\underline{z}+\underline{x}}{\underline{j}}}$. As this is clearly $C^\infty$, so is $\vartheta$. Therefore $\mf{G}\rtimes_\vartheta \mb{R}^r$ is a well-defined connected and simply-connected Lie group.
\end{proof}

We can now present the main nilmanifold construction that we shall use in Proposition \ref{lem:growing-nilmanifold}.

\begin{proposition}\label{lem:growing-nilmanifold}
Let $k,r\in \mb{N}$ and let $(G,G_\bullet)$ be a degree-$k$ prefiltration of connected, simply-connected Lie groups $(G_i)_{i\in[k]}$. Let $H:=\hom(\mc{D}_1(\mb{Z}^r),G_\bullet)\rtimes_\vartheta \mb{R}^{r}$, let  $H_0=H_1=H$ and for $i\geq 2$ let $H_i:=\hom(\mc{D}_1(\mb{Z}^r),G_\bullet^{+(i-1)})\rtimes_\vartheta \{\underline{0}\}$. Then $H_\bullet=(H_i)_{i\geq 0}$ is a degree-$(k+1)$ filtration of connected and simply-connected Lie groups. Moreover, letting $\Gamma$ be a discrete cocompact subgroup of $G$ such that $G_i$ is rational in $G$ relative to $\Gamma$ for all $i$, letting $\Gamma_\bullet$ be the induced filtration, and letting $\Lambda:=\hom(\mc{D}_1(\mb{Z}^r),\Gamma_\bullet)\rtimes_\vartheta \mb{Z}^{r}$, we have that $H/\Lambda$ is a filtered nilmanifold of degree $k+1$.
\end{proposition}
To prove this we shall use the following generalization of Theorem 6 from \cite[p.\ 238]{HKbook}.

\begin{lemma}\label{lem:thm6-hk}
Let\footnote{To align better with the notation of the proof of \cite[Theorem 6]{HKbook}, in this theorem the degree of the filtration will be $s$ instead of the usual $k$. The symbol $k$ will be used in the proof for other matters.} $s,r\in \mb{N}$ and let $(G,G_\bullet)$ be a prefiltration of connected, simply-connected Lie groups $(G_i)_{i\in[s]}$. Then for every $i\ge 0$, the set $\hom(\mc{D}_1(\mb{Z}^r),G_\bullet^{+i})$ with point-wise multiplication is a subgroup of $G^{\mb{Z}^r}$. Furthermore, these groups form a prefiltration and for every $\underline{x}\in \mb{R}^r$ we have that $\partial_{\underline{x}}(\hom(\mc{D}_1(\mb{Z}^r),G_\bullet^{+i}))\subset \hom(\mc{D}_1(\mb{Z}^r),G_\bullet^{+(i+1)})$ where $\partial_{\underline{x}}(g):=g(\cdot + \underline{x})g^{-1}$.
\end{lemma}

\begin{proof}
We argue essentially as in the proof of \cite[Ch. 14 Theorem 6]{HKbook} with some small adjustments. In our case, we replace  the definition of $H(i,m)$ from \cite[p. 238]{HKbook} with
\[
H(i,m)=\left\{\varphi \in G^{\mathbb{Z}^r} : \text{For } 0 \le k \le m \text{ and } \underline{x}_1,\ldots,\underline{x}_k \in \mb{R}^{r}\text{ we have }\partial_{\underline{x}_1} \cdots \partial_{\underline{x}_r} (\varphi) \in G_{i+k}^{\mathbb{Z}^r}\right\}.
\]
Note that the shift $\vartheta_{\underline{x}}:g\mapsto g(\cdot+\underline{x})$ is well-defined by Lemma \ref{lem:generalized-shift}. The rest of the proof works  the same as the one in \cite[p. 238]{HKbook}  replacing $\partial$ by $\partial_{\underline{x}}$ and $\sigma(\phi)$ by $\vartheta_{\underline{x}}(\phi)$ for some $\underline{x}\in \mb{R}^r$.
\end{proof}

\begin{proof}[Proof of Proposition \ref{lem:growing-nilmanifold}]
We first prove that $H_\bullet$ is a filtration of degree $k+1$. Let $\mf{G}=\hom(\mc{D}_1(\mb{Z}^r),G_\bullet)$, let $\mf{G}_0=\mf{G}_1=\mf{G}$ and for  $i\geq 2$ let $\mf{G}_i=\hom(\mc{D}_1(\mb{Z}^r),G_\bullet^{+(i-1)})$. Recall that $H=H_0=H_1=\mf{G}\rtimes_{\vartheta}\mb{R}^r$ and that $H_i:=\mf{G}_i\times \{\underline{0}\}$ for $i\ge 2$, and note that $\mf{G}$ equals the product set $(\mf{G}\times \{\underline{0}\}) \cdot (\{\id\}\times \mb{R}^r)$.

We need to show that for all $i,j\ge 1$ we have $[H_i,H_j]\subset H_{i+j}$. We claim first that the groups $H_i$ for $i\ge 2$ are normal subgroups of $H$. To see this, let $h$ be any element of $H$ and write $h=ab$ where $a=(\id,x)\in \{\id\}\times \mb{R}^{r}$ and $b\in \mf{G}\times \{\underline{0}\}$, and for any $i\ge 2$ let $c\in H_i$. We need to show that $c^{ab}\in H_i$. We have $c^{ab}=b^{-1}a^{-1}cab=(c^a)^b$ (using the notation $u^v$ for the conjugate $v^{-1}uv$). Now note first that  $c^a\in H_i$ because, since $i\geq 2$, we have that $c=(\varphi,0)$ and so $c^a=(\id,-x)(\varphi,0)(\id,x)=(\id,-x)(\varphi,x)=(\vartheta_{-x}(\varphi),0)\in H_i$. Then by Lemma \ref{lem:thm6-hk}, since $b=(\varphi',0)$, we have $(c^a)^b\in H_i$ by the normality of the subgroups $\mf{G}_i$ in $\mf{G}$ (since $(\mf{G}_i)_{i\geq 0}$ is a filtration by Lemma \ref{lem:thm6-hk}).

Using the identity $[ab,c]=[a,c]^b[b,c]$ (which holds for any elements $a,b,c$ in any group, where we are using $[a,b]:=a^{-1}b^{-1}ab$) and the normality of the groups $H_i$ for $i\ge 2$, we see that to check the desired inclusion $[H_i,H_j]\subset H_{i+j}$ it suffices to check that commutators of generators of $H_i,H_j$ lie in  $H_{i+j}$. The cases where $i,j\ge 2$ follow directly from Lemma \ref{lem:thm6-hk} (again since this gives us that $(\mf{G}_i)_{i\geq 0}$ is a filtration).

If $i=1$ and $j>1$ it suffices to check that $[a,b]$ lies in the correct term in the filtration in the cases where $a\in H_i$ is either in $\mf{G}\times\{\underline{0}\}$ or in $\{\id\}\times \mb{R}^{r}$ (as these sets together generate $H$) and $b=(\varphi,0)$. In the former case the result follows by Lemma \ref{lem:thm6-hk}, and in the latter case it follows by the fact that then the commutator with $a$ corresponds to an application of $\partial_{x}$ to $\varphi$ (where $a=(\id,x)$, which results in $\partial_x\varphi\in H_{j+1}$ as required. (For $i>1$, $j=1$ the argument is similar using the alternative formula $[a,bc]=[a,c][a,b]^c$). Finally, if $i=j=1$, note that $[ab,cd]=[a,cd]^{b}[b,cd]=([a,d][a,c]^d)^b[b,d][b,c]^d$, then the only case that is different from the previous ones is when both elements lie in $ \{\id\}\rtimes_\vartheta \mb{R}^{r}$. But in this case we clearly have that the commutator is the identity and the result follows.

The subgroup $\Lambda$ is seen to be discrete using the fact that $\Gamma$ is discrete. To see that for every $i\in[k+1]$ the group $H_i$ is a rational subgroup of $H$ with respect to $\Lambda$, note first that since $\Gamma_i=\Gamma\cap G_i$ is by assumption cocompact in $G_i$, we have that $\hom(\mc{D}_1(\mb{Z}^r),G_i/\Gamma_i)$ is a compact set, and since this is homeomorphic to $\mf{G}_i/\mf{\Gamma}_i$ (where $\mf{\Gamma}_i:=\mf{\Gamma}\cap \mf{G}_i$), we conclude that $\mf{\Gamma}_i$ is cocompact in $\mf{G}_i$, so there is a compact set $K_i\subset \mf{G}_i$ such that $K_i\mf{\Gamma}_i=\mf{G}_i$ and then, letting $S_i=[0,1]^r$ for $i=0,1$ and $S_i=\{\underline{0}\}$ otherwise, we have that $K_i\times S_i$ is a compact subset of $H_i$ such that $(K_i\times S_i)\cdot \Lambda_i=H_i$, so $\Lambda_i$ is cocompact in $H_i$ as required.

Finally, note that by Lemma \ref{lem:generalized-shift}  all the $H_i$ are connected and simply-connected.
\end{proof}
\noindent Note that for a connected simply-connected nilpotent Lie group $H$, and elements $h_1,\ldots,h_r\in H$ we have that these elements commute pairwise if and only if the unique 1-parameter subgroups $h_i^t$ commute pairwise (meaning that for any $i,j\in [r]$ and $t_i,t_j\in \mb{R}$, the elements $h_i^{t_i},h_j^{t_j}$ commute). The only if direction can be seen using the Baker-Campbell-Hausdorff formula.

We can now finally obtain our main tool to extend nilsequences in the non-split case. 

\begin{proposition}\label{prop:exte-non-split}
Let $k,r\in \mb{N}$. Then, having fixed a notion of complexity for all nilmanifolds, there exists functions $W:\mb{N}^2\to\mb{N}$ and $Q:\mb{N}^2\times\mb{R}\to\mb{R}$ such that the following holds.

Let $G/\Gamma$ be a degree-$k$ filtered nilmanifold of complexity at most $m$ where $G$ is connected and simply-connected, let $\ab\cong \prod_{i=1}^r \mb{Z}_{n_i}$ be a finite abelian group, let $\varphi:\mc{D}_1(\ab)\to G$ be a morphism, and let $F:G/\Gamma\to \mb{C}$ be a Lipschitz function with Lipschitz norm at most $M$. Then there exists a degree-$k$ filtered nilmanifold $H/\Lambda$ of complexity at most $W(m,r)$, with $H$ connected and simply-connected, and pairwise commuting elements $h_1,\ldots,h_r\in H$ with $h_i^{n_i}\in \Lambda$ for $i\in[r]$, and a $Q(m,r,M)$-Lipschitz function $F':H/\Lambda\to \mb{C}$, such that
\begin{equation}\label{eq:good-repre}
F(\varphi(\underline{z})\Gamma)=F'(\prod_{i\in[r]}h_i^{z_i}\Lambda)\quad \text{for all }\underline{z}\in\prod_{i\in[r]}\mb{Z}_{n_i}\cong  \ab. 
\end{equation}
\end{proposition}
\noindent This is a multivariable extension of \cite[Proposition C.2]{GTZ} (the latter is the 1-variable case $r=1$).
\begin{proof}
Let $H',\Lambda'$ be as given by Proposition \ref{lem:growing-nilmanifold} applied with $G_\bullet$, and let $H,\Lambda$ be as given by that proposition applied with $G_\bullet^{+1}$. Note that $H'/\Lambda'$ is filtered of degree $k+1$, whereas the subnilmanifold $H/\Lambda$ is filtered of degree $k$. Note also that, by construction, if $G/\Gamma$ has complexity $m$ then $H/\Lambda$ has complexity at most $W(m,r)$ for some function $W$.

By definition $H'=\mf{G}'\rtimes_\vartheta \mb{R}^r$ where $\mf{G}'=\hom(\mc{D}_1(\mb{Z}^r),G_\bullet)$. We embed $\mf{G}'$ as the subgroup $\mf{G}'\times\{\underline{0}\}\leq  H'$ and we embed $\mb{R}^r$ as the subgroup $\{\id\}\times \mb{R}^r\leq H'$.  Abusing the notation, we assume that $\varphi$ is a polynomial map in $\hom(\mc{D}_1(\mb{Z}^r), G_\bullet)$ which is $n_i$-periodic in the $i$-th variable, for each $i\in[r]$.

Let $K\subset H'$ be a compact subset such that $K\Lambda'=H'$ and let $(\varphi',\underline{x'})\in K$ be such that $(\varphi',\underline{x'})\Lambda'=(\varphi,\underline{0})\Lambda'$. Let $\{e_i:i\in [r]\}$ be the standard basis of $\mb{R}^r$, and let $h_i:=(\varphi',\underline{x'})^{-1}(\id,e_i)(\varphi',\underline{x'})$. Clearly the one-parameter subgroup corresponding to this element is $h_i^t:=(\varphi',\underline{x'})^{-1}(\id,t e_i)(\varphi',\underline{x'})$ for $t\in \mb{R}$ (where $t e_i$ is the scalar multiple of $e_i$ by $t$). Since $H$ is normal in $H'$ and $e_i\in H$, we have that $h_i\in H$ for $i\in[r]$. Using the definition of product in the group $H'$ and the fact that $\mb{R}^r$ is commutative, it follows that the 1-parameter subgroups $h_i^t$, $i\in [r]$ commute. For $i\in[r]$, as $\varphi$ is $n_i$-periodic in the $i$-variable, we have that $(\id,n_ie_i)(\varphi,\underline{0})(\id,-n_ie_i)= (\varphi(\cdot+n_ie_i),e_i)(\id,-n_ie_i)=(\varphi,\underline{0})$. Hence (recalling that $\Lambda'=\hom(\mc{D}_1(\mb{Z}^r),\Gamma_\bullet)\rtimes_\vartheta\mb{Z}^r$ and therefore $(\id,-n_ie_i)\in \Lambda'$) we have $(\id,n_ie_i)(\varphi,\underline{0})\Lambda'=(\varphi,\underline{0})\Lambda'$. Therefore $(\id,n_ie_i)(\varphi',\underline{x'})\Lambda'=(\varphi',\underline{x'})\Lambda'$ and it follows that $h_i^{n_i}\in \Lambda'$.

Let us now construct the function $F'$. First, note that left multiplication by any element $h'\in H'$ on the nilmanifold $H'/\Lambda'$ is Lipschitz. As $K$ is compact, the Lipchitz constant of left-multiplication $L_{h'}$ by any element $h'\in K$ is uniformly bounded by some constant, i.e.\ $\sup_{h'\in K}\|L_{h'}\|_L=O_{K,H'/\Lambda'}(1)$. As $K$ can be chosen fixed for $H'/\Lambda'$, we can assume that this constant is simply $O_{H'/\Lambda'}(1)$. As the construction of $H'/\Lambda'$ depends solely on $G/\Gamma$ and $r$, this constant is bounded by some quantity $O_{m,r}(1)$.

Following the proof of \cite[Proposition C.2]{GTZ}, let $\rho:H'/\Lambda'\to \mb{T}^r$ be the morphism that sends $(g,\underline{x})\Lambda'\mapsto \underline{x}\mod \mb{Z}^r$. Then $\rho^{-1}(\{\underline{0}\})$ is a compact subset. Moreover, the map $\mathrm{ev}_0:\rho^{-1}(\{\underline{0}\})\to G/\Gamma$ given by\footnote{Note that any element in $\rho^{-1}(\{\underline{0}\})$ has a representative of the form $(g,\underline{0})$ and it easily follows that $\mathrm{ev}_0$ is well-defined.} $(g,\underline{0})\Lambda'\mapsto g(\underline{0})\Gamma$ is a Lipchitz map with $\|\mathrm{ev}_0\|_L=O_{G/\Gamma}(1)$. Similarly as in \cite[proof of Theorem 7.16]{CGSS-doucos}, which is in essence Kirszbraun’s theorem, we can extend the map $F\co \mathrm{ev}_0$ to a Lipschitz map on $H'/\Lambda'$ in such a way that the Lipschitz constant is $O_{G/\Lambda,M}(1)$. Abusing the notation, we denote this function\footnote{Moreover, we can assume that $F\co \mathrm{ev}_0$ is also 1-bounded simply by composing with $t\mapsto t\text{ if }|t|\le 1\text{ and }t/|t|\text{ otherwise}$. To avoid complicating further the notation, we omit this step and we simply assume that $F\co \mathrm{ev}_0$ is 1-bounded and $\|F\co \mathrm{ev}_0\|_L=O_{G/\Gamma}(1)$.} by $F\co \mathrm{ev}_0$. Finally, note that $F\co \mathrm{ev}_0\co L_{\varphi'}$ is then a 1-bounded Lipschitz map with Lipschitz norm $\|F\co \mathrm{ev}_0\co L_{\varphi'}\|_L\le \|F\co \mathrm{ev}_0\|_L\|L_{\varphi'}\|_L=O_{G/\Gamma,M}(1)$. This map is defined on $H'/\Lambda'$, but we may simply restrict it to the subnilmanifold $H/\Lambda$ (with possibly larger Lipschitz constant depending only on $G/\Gamma$ by \cite[Lemma A.17]{GTorb}). Hence, we shall assume that $F':=F\co \mathrm{ev}_0\co L_{\varphi'}$ is a $Q(m,r,M)$-Lipschitz map $F':H/\Lambda\to \mb{C}$.

To see that we have indeed the equality \eqref{eq:good-repre}, we have to unfold the definition of $F'$ and the elements $h_i$. Indeed note that $F'(\prod_{i\in[r]}h_i^{z_i}\Lambda)=F\co \mathrm{ev}_0\co L_{\varphi'}(\prod_{i\in[r]}h_i^{z_i}\Lambda') = F\co \mathrm{ev}_0((\id,\sum_{i\in[r]}z_i\underline{e_i})(\varphi',\underline{x'})\Lambda') = F\co \mathrm{ev}_0((\id,\underline{z})(\varphi,\underline{0})\Lambda') = F\co \mathrm{ev}_0((\varphi(\cdot+\underline{z}),\underline{z})\Lambda') = F\co \mathrm{ev}_0((\varphi(\cdot+\underline{z}),\underline{z})(\id,-\underline{z})\Lambda') = F\co \mathrm{ev}_0((\varphi(\cdot+\underline{z}),\underline{0})\Lambda') =  F(\varphi(\underline{z})\Gamma)$.\end{proof}

\noindent It is now straightforward to achieve the desired extension in the non-split case, as follows.

\begin{corollary}[Extending in the non-split case]\label{cor:non-split-case}
Let $k,d,p,r\in \mb{N}$, let $A$ be a finite abelian group of rank $r-1$, and let $(G/\Gamma,G_\bullet)$ be a filtered nilmanifold of degree $k$ where the nilpotent Lie group $G$ is connected and simply-connected. Then for every nilsequence $F(\varphi(\cdot)\Gamma):(p\mb{Z}_{pd})\oplus A\to\mb{C}$ there exists another nilsequence $F'(\phi(\cdot)\Lambda):\mb{Z}_{pd}\oplus A\to\mb{C}$ such that
\begin{equation}\label{eq:corext}
\forall\, \underline{z}\in (p\cdot \mb{Z}_{pd})\oplus A,\quad F(\varphi(\underline{z})\Gamma)=F'(\phi(\underline{z})\Lambda).
\end{equation}
Moreover, if for some $m\ge 0$ we have that the nilsequence $F(\varphi(\cdot)\Gamma)$ has complexity $m$, then $F'(\phi(\cdot)\Lambda)$ has complexity at most $K(m,r)$ for some fixed function $K:\mb{Z}_{\ge 0}^2\to \mb{Z}_{\ge 0}$.
\end{corollary}
\begin{proof}
We can assume that the finite abelian group $\ab:=(p\mb{Z}_{pd})\times A$ has an expression $\ab =\prod_{i=1}^r\mb{Z}_{n_i}$ where $n_1$ is the order $d$ of the first component $p\mb{Z}_{pd}\cong \mb{Z}_d$ of $\ab$.

We now gather the elements given by Proposition \ref{prop:exte-non-split}: let $H/\Lambda$ be the nilmanifold, let $(h_i^t)_{i\in [r]}$ be the commuting 1-parameter subgroups of $H$, and let $F':H/\Lambda:\to\mb{C}$ be the 1-bounded Lipschitz function, such that \eqref{eq:good-repre} holds.

Now note that in the 1-parameter subgroup $h_1$ there is a (unique) $p$-th root of $h_1$, namely an element $w_1$ such that $w_1^p=h_1$. Define the polynomial map $\phi(\cdot)\Lambda$ on the group $\ab':= \mb{Z}_{pd}\times A=\mb{Z}_{pd}\times \prod_{i=2}^r\mb{Z}_{n_i}$ by $\phi(\underline{z})\Lambda = w_1^{z_1}h_2^{z_2}\cdots h_r^{z_r}\Lambda$, noting that this is a well-defined map on $\ab'$ thanks to the fact that $w_1^{pd}=h_1^d\in \Lambda$ and the 1-parameter subgroups $w_1^t,h_2^t,\ldots,h_r^t$ commute. From $w_1^p=h_1$ we then deduce that for every $\underline{z}\in \ab$ we have $\phi(\cdot)\Lambda = \prod_{i=1}^r h_i^{z_i} \Lambda$. Combining this with \eqref{eq:good-repre}, the desired equation \eqref{eq:corext} follows.
\end{proof}

We close this section by combining all the above results to obtain the following extension theorem, which is the one that we shall use in the next section.

\begin{theorem}\label{thm:mainextension}
Let $\ab$ be a finite abelian group of rank $r$, let $\ab_0$ be a subgroup of $\ab$, and let $F_0(g_0(\cdot)\Gamma_0)$ be a polynomial nilsequence of degree $k$ and complexity at most $C_0$ on $\ab_0$. Then there exists a nilsequence $F(g(\cdot)\Gamma)$ on $\ab$ of complexity at most $C=C(C_0,r,k)$ such that for every $z\in \ab_0$ we have $F(g(z)\Gamma)=F_0(g_0(z)\Gamma_0)$.
\end{theorem}

\begin{proof}
Take the sequence of subgroups $\ab_0\leq \ab_1\leq \cdots\leq \ab_t=\ab$ given in \eqref{eq:ladder}. For each $i$, if $\ab_{i-1}$ is a direct summand of $\ab_i$, assuming we have extended $F_0(g_0(\cdot)\Gamma_0)$ already to a nilsequence $F_{i-1}(g_{i-1}(\cdot)\Gamma_{i-1})$ on $\ab_{i-1}$, we apply Lemma \ref{lem:splitcase} and thus extend the latter nilsequence to $F_i(g_i(\cdot)\Gamma_i)$ where in fact $G_i/\Gamma_i=G_{i-1}/\Gamma_{i-1}$ and $F_i=F_{i-1}$. If on the contrary $\ab_i$ is a non-split extension of $\ab_{i-1}$, then we apply Corollary \ref{cor:non-split-case} to obtain the extension $F_i(g_i(\cdot)\Gamma_i)$. Continuing this way iteratively, we eventually obtain the claimed extension  $F(g(\cdot)\Gamma)$ on $\ab$. 

It remains to prove that the complexity of $F(g(\cdot)\Gamma)$ is controlled. For each $i\in [0,t]$, let $C_i$ be a complexity bound for the nilsequence $F_i(g_i(\cdot)\Gamma_i)$ obtained at the $i$-th extension. Recall from Lemma \ref{lem:ladder} that $t\leq \log_2(|\ab|/|\ab_0|)$. Since the index $|\ab|/|\ab_0|$ is equal to the order of the finite abelian group $H$ such that the 1-step factor of $G_0/\Gamma_0$ is isomorphic to $H\times \mb{T}^\ell$, we have $t\leq \log_2(C_0)$. For each $i$, if at the $i$-th iteration $\ab_i$ is a split extension of $\ab_{i-1}$, then we can take $C_i=C_{i-1}$ (since the nilmanifold $G_i/\Gamma_i$ and function $F_i$ are the same as the previous ones). If on the contrary the $i$-th extension requires applying Corollary \ref{cor:non-split-case}, then $C_i\leq K(C_{i-1},r)$ (for the function $K$ in the corollary). It follows that $C=C_t\leq K^{(t)}(C_0,r)$, where the superscript $(t)$ here indicates $t$ iterations of the function $K(\cdot,r)$.
\end{proof}

\section{Proof of the inverse theorem}\label{sec:mainproof}

In this section we prove our main result, the inverse theorem for abelian groups $\ab$ of bounded rank (Theorem \ref{thm:bndrankinv-intro}), which we recall here for convenience. 

\begin{theorem}\label{thm:bndrankinv}
For every $k,R\in \mb{N}$ and $\delta>0$, there exist  $C>0$ and $\varepsilon>0$ such that the following holds. Let $\ab$ be a finite abelian group of rank at most $R$, and let $f:\ab\to\mb{C}$ be a 1-bounded function satisfying $\|f\|_{U^{k+1}}\geq \delta$. Then there exists a connected and simply-connected filtered Lie group $(G,G_\bullet)$ of degree at most $k$, and a discrete cocompact subgroup $\Gamma\leq G$, such that the toral nilspace associated with the filtered nilmanifold $(G/\Gamma,G_\bullet)$ has complexity at most $C$, and there is a polynomial map $g:\ab\to G/\Gamma$, and a continuous 1-bounded function $F:G/\Gamma\to\mb{C}$ of Lipschitz constant at most $C$, such that
\begin{equation}\label{eq:invlobound}
|\mb{E}_{x\in \ab} f(x)\overline{F(g(x))}| \geq \varepsilon.  
\end{equation}
\end{theorem}

\begin{proof}
By Proposition \ref{prop:mainstep1}, we have a connected filtered nilmanifold $G_0/\Gamma_0$ of degree $k$ and complexity at most $C_0(\delta,k,R)$, a subgroup $\ab_0\leq \ab$ with index $|\ab|/|\ab_0|=O_{C_0}(1)$, and a nilsequence $F_0(g_0(\cdot)\Gamma_0)$ on $\ab_0$ where $F_0:G_0/\Gamma_0\to\mb{C}$ has Lipschitz constant at most $C_0$ and $g_0(\cdot)\Gamma_0\in \hom(\mc{D}_1(\ab_0),(G_0/\Gamma_0,(G_{0\bullet}))$, such that for some $t_0\in \ab$ we have
\begin{equation}\label{eq:startingbnd}
  \mb{E}_{x\in \ab_0} f(x+t_0)F_0(g_0(x)\Gamma_0) \geq \varepsilon_0. 
\end{equation}
Let $F'(g'(\cdot)\Gamma')$ be the nilsequence on $\ab$ extending $F_0(g_0(\cdot)\Gamma_0)$, given by Theorem \ref{thm:mainextension}.

By basic Fourier analysis, the indicator function $1_{\ab_0}$ on $\ab$ satisfies (noting that $\frac{1}{|\ab_0^\perp|}=\frac{|\ab_0|}{|\ab|}$) $ 1_{\ab_0}(x)=\sum_{\chi\in \ab_0^\perp}\tfrac{|\ab_0|}{|\ab|} \chi(x)=\mb{E}_{\chi\in \ab_0^\perp}\chi(x)$. Hence, starting from \eqref{eq:startingbnd}, we have
\begin{align*}
& \varepsilon_0  \leq \mb{E}_{x\in \ab_0}f(x+t_0) F_0(g_0(x)\Gamma_0) =  \tfrac{|\ab|}{|\ab_0|} \mb{E}_{x\in \ab} f(x+t_0) F'(g'(x)\Gamma') 1_{\ab_0}(x)\\
&  = \tfrac{|\ab|}{|\ab_0|} \mb{E}_{\chi\in \ab_0^\perp} \mb{E}_{x\in \ab}f(x+t_0) F'(g'(x)\Gamma') \chi(x) \leq  \tfrac{|\ab|}{|\ab_0|} \max_{\chi\in \ab_0^\perp} | \mb{E}_{x\in \ab}f(x) F'(g'(x-t_0)\Gamma') \chi(x-t_0)|.
\end{align*}
Now fix any $\chi\in \ab_0^\perp$ attaining this last maximum, let us relabel the shifted polynomial map $x\mapsto g'(x-t_0)$ as a polynomial map $g''(x)$, and let us ignore the multiplicative constant $ \chi(t_0)$ of modulus 1, thus concluding that
\begin{equation}\label{eq:correl2}
 | \mb{E}_{x\in \ab}f(x) F'(g''(x)\Gamma') \chi(x)| \geq \varepsilon_0 \tfrac{|\ab_0|}{|\ab|} =\Omega_{C_0}(\varepsilon_0)=:\varepsilon.
\end{equation}
By the standard representation of characters via non-degenerate symmetric bilinear forms (see \cite[Ch.\ 4]{T-V}), there exists such a form $\cdot:\ab\times\ab\to \mb{T}$ and some $\xi\in \ab$ such that $\chi(x)=e(\xi\cdot x)$. We can now define the nilmanifold $G/\Gamma$ as $(G=G'\times \mb{R})/(\Gamma=\Gamma' \times \mb{Z})$, which is isomorphic as a compact nilspace to $(G'/\Gamma')\times \mb{T}$, and define the  map $g(\cdot)\Gamma:\ab\to G/\Gamma$ by $g(x)\Gamma:= (g''(x)\Gamma',\xi\cdot x)$, which is indeed a nilspace morphism $\mc{D}_1(\ab)\to (G'/\Gamma')\times \mc{D}_1(\mb{T})$. We define also the 1-bounded Lipschitz function $F:G/\Gamma\to\mb{C}$ by $F(x\Gamma',z)=F'(x\Gamma')e(z)$ for any $x\Gamma'\in G'/\Gamma'$ and $z\in\mb{T}$, and it can be checked easily that $F$ has bounded Lipschitz constant if $F_0$ does (using that both $F_0$ and $e(\cdot)$ are 1-bounded). Thus, the inequality \eqref{eq:correl2} is rewritten as $
|\mb{E}_{x\in \ab}f(x) F(g(x)\Gamma)| \geq \varepsilon$. This completes the proof.\end{proof}

\end{document}